\documentclass[amscd,amssymb,verbatim,10pt]{amsart}
\usepackage[fontsize = 10bp]{fontsize}

\usepackage{amssymb, latexsym, amsmath, amscd, array, graphicx}
\usepackage{hyperref}
\usepackage[all]{xy}
\usepackage[mathscr]{euscript}

\swapnumbers \numberwithin{equation}{section}

\usepackage{color}
\hypersetup{
    colorlinks=true,
    linkcolor=red,
    urlcolor=black,
    linktoc=all,
    citecolor=blue
           }
\usepackage{tikz-cd}
\usepackage{bbm,graphicx}

\theoremstyle{plain}

\newtheorem{thm}{Theorem}[section]
\newtheorem{theorem}[thm]{Theorem}

\newtheorem{lemma}[thm]{Lemma}
\newtheorem{conjec}[thm]{Conjecture}
\newtheorem{prop}[thm]{Proposition}
\newtheorem{cor}[thm]{Corollary}

\theoremstyle{definition}
\newtheorem{defn}[thm]{Definition}

\newtheorem{remark}[thm]{Remark}

\newtheorem{ex}[thm]{Example}

\DeclareMathOperator{\cat}{\mathsf{cat}}
\DeclareMathOperator{\TC}{\mathsf{TC}}
\DeclareMathOperator{\dcat}{\mathsf{dcat}}
\DeclareMathOperator{\dTC}{\mathsf{dTC}}
\DeclareMathOperator{\secat}{\mathsf{secat}}
\DeclareMathOperator{\dsecat}{\mathsf{dsecat}}

\DeclareMathOperator{\cd}{{\rm cd}}

\DeclareMathOperator{\Ker}{{\rm Ker}}

\DeclareMathOperator{\supp}{{\rm supp}}
\DeclareMathOperator{\ATC}{{\rm ATC}}
\DeclareMathOperator{\acat}{{\rm acat}}

\newcommand \pa[2]{\frac{\partial #1}{\partial #2}}

\def\scr{\mathcal}

\def\B{{\scr B}}

\def\C{{\mathbb C}}
\def\Z{{\mathbb Z}}
\def\Q{{\mathbb Q}}
\def\R{{\mathbb R}}

\def\1{\hbox{\rm\rlap {1}\hskip.03in{\rom I}}}
\def\Bbbone{{\rm1\mathchoice{\kern-0.25em}{\kern-0.25em}
{\kern-0.2em}{\kern-0.2em}I}}

\def\pa{\partial}

\def\wt{\widetilde}
\def\wh{\widehat}

\def\ov{\overline}

\long\def\forget#1\forgotten{} %

\newcommand\ver[1]{\marginpar{\tiny Changed in Ver \VER}}

\date{\today}

\begin{document}

\title[Sequential Distributional Topological Complexity]{On Sequential Versions of Distributional Topological Complexity}

\author[Ekansh Jauhari]{Ekansh Jauhari}

\address{Ekansh Jauhari, Department of Mathematics, University
of Florida, 358 Little Hall, Gainesville, FL 32611-8105, USA.}
\email{ekanshjauhari@ufl.edu}

\subjclass[2020]
{Primary 55M30; Secondary 68T40, 70B15.}

\keywords{}

\begin{abstract}
We define a (non-decreasing) sequence $\{\dTC_m(X)\}_{m\ge 2}$ of sequential versions of distributional topological complexity ($\dTC$) of a space $X$ introduced by Dranishnikov and Jauhari~\cite{DJ}. This sequence generalizes $\dTC(X)$ in the sense that $\dTC_2(X) = \dTC(X)$, and is a direct analog to the well-known sequence $\{\TC_m(X)\}_{m\ge 2}$. We show that like $\TC_m$ and $\dTC$, the sequential versions $\dTC_m$ are also homotopy invariants. Furthermore, $\dTC_m(X)$ relates with the distributional LS-category ($\dcat$) of products of $X$ in the same way as $\TC_m(X)$ relates with the classical LS-category ($\cat$) of products of $X$. On one hand, we show that in general, $\dTC_m$ is a different concept than $\TC_m$ for each $m \ge 2$. On the other hand, by finding sharp cohomological lower bounds to $\dTC_m(X)$, we provide various examples of closed manifolds $X$ for which the sequences $\{\TC_m(X)\}_{m\ge 2}$ and $\{\dTC_m(X)\}_{m\ge 2}$ coincide.  
\end{abstract}




\keywords{Sequential distributional topological complexity, distributed navigation algorithm,  distributional sectional category, distributional Lusternik--Schnirelmann category, sequential topological complexity}

\maketitle

\section{Introduction}\label{Introduction}
In robotics, an autonomously functioning mechanical system, such as a robot, is typically required to move from a specified initial position to a specified desired position inside a configuration space in an optimal way. By optimal, we mean with the least possible\footnote{\hspace{0.5mm} The least possible number of discontinuities is determined by the topology of the configuration space in hand.} number of instabilities or discontinuities of motion. depending on the topology of the configuration space. This is a motion planning problem whose inputs are the positions and output is the motion of the system between them. In many practical situations, to obtain a more precise and controlled motion, such a system is additionally required to pass through a fixed number of intermediate positions at some specified timestamps. This adds more constraints to the motion planning problem.

This paper is motivated by the following sequential motion planning problem. Given an autonomous mechanical system with configuration space $X$, a number $m \ge 2$, and positions $x_1, x_2, \ldots, x_m\in X$, we want to construct a continuous algorithm to get from $x_1$ to $x_m$ via the $m-2$ intermediate positions $x_2, \ldots, x_{m-1}$ attained in that order for each ordered tuple $(x_1,x_2, \ldots, x_m) \in X^m$. While studying this problem for $m = 2$, M.~Farber~\cite{Far1} introduced the notion of the {\em topological complexity} of a space $X$, denoted  $\TC(X)$, as the minimal degree of instability of motion planning algorithms for systems whose configuration space is $X$. This notion was generalized for each $m \ge 2$ in a natural way by Y.~B.~Rudyak~\cite{Ru} to the $m$-th {\em sequential topological complexity} of the space, denoted $\TC_m(X)$. 

In a recent joint work with A. Dranishnikov~\cite{DJ}, we considered this problem for the case $m=2$ for some advanced autonomous systems that can break into a finite number of pieces at the initial position so that all the pieces travel independently to the desired final position where they reassemble back into the system. An example of such a system is the robot ``Terminator 2" from the eponymous movie. For such systems, the notion of the {\em distributional topological complexity} of a space $X$, denoted $\dTC(X)$, was introduced in~\cite{DJ} as the minimal number of pieces into which the system needs to break to be able to perform a continuous motion between all possible pairs of positions in the configuration space $X$. For some advanced systems that are capable of breaking and reassembling, depending on the configuration space of the system, $\dTC$ could an improvement of $\TC$. In this paper, we consider the generalized problem for $m \ge 2$ and introduce a natural generalization of $\dTC$, namely the $m$-th {\em sequential distributional topological complexity}, denoted $\dTC_m$, which, we think, offers a better solution to the above sequential motion planning problem for some advanced autonomous systems.

\subsection{Planning a sequential motion}
Consider an advanced autonomous system that is required to reach a position $x_{i} \in X$ at a time $t_{i} \in [0,1]$ for each $1 \le i \le m$, where $t_i < t_{i+1}$. The idea for the sequential motion of the system is as follows.
\vspace{-1.5mm}
\begin{enumerate}
        \itemsep-0.35em 
\item At time $t_1$, the system breaks into finitely many weighted pieces at the initial position $x_1$. All the pieces travel independently for time $t_2 - t_1$ to reach position $x_2$ at time $t_2$. 
\item At time $t_2$ and position $x_2$, all the pieces reassemble back into the system. Then, the system breaks again at $x_2$ into the same number of pieces having the same weights. These pieces travel to reach $x_3$ at time $t_3$. 
\item The process continues like this, where the system breaks into the same number of pieces having the same weights at time $t_i$ and position $x_i$, the pieces travel independently to reach $x_{i+1}$ at $t_{i+1}$, where they reassemble into the system, and then the system breaks again in the same manner.
\item Finally, the pieces reach the final position $x_m$ at time $t_m$ where they reassemble into the system at last.
\end{enumerate} 
We assume that our system above has weight $1$, i.e., the sum of the non-negative weights of the pieces involved is $1$. 

The ``weights of the pieces" represent the percentage of the portion of the original system in the corresponding pieces. For example, if the weights are $0.23, 0.16,$ and $0.61$, then the system breaks into pieces whose capacities/weights are $23\%, 16\%$, and $61\%$ of its original capacity/weight. In another sense, the ``weight of a piece" represents the \textit{probability} that at a given instance, the system is moving/traveling through that particular piece. More precisely, if we regard each piece $x_i$ of the system as a ``state" of its motion, then the weight $\lambda_i$ is the probability that the system is in state $x_i$. For a related probabilistic interpreptation in the case of $\TC$, we refer to the reader to~\cite[Section 12]{Far2}.


The input for our system is an ordered $m$-tuple of positions. Note that we require our system to first analyze the $m$ positions and then decide the weighted pieces it needs for a continuous motion between those $m$ positions in the given order by the process of repetitive breaking and reassembling. This is different, less chaotic, and more convenient than letting the system decide its weighted pieces at each individual $t_i$ to travel between $x_i$ and $x_{i+1}$ for each $1 \le i \le m-1$.

Let $n \ge 1$ be the maximum number of pieces that the system can break into while traveling between all possible ordered $m$-tuples of positions in $X$. If no such $n$ exists, we conclude that the system is extremely complicated for practical purposes and we say that its corresponding sequential distributional topological complexity is infinite. Such a case is hardly found in real-world situations and is therefore less interesting, both from a physical as well as a theoretical point of view.

\subsection{Continuous motion planning algorithm}\label{contmot}
Let $\ov{x} = (x_1,\ldots,x_m) \in X^m$ be an ordered $m$-tuple. We define an {\em $n$-distributed $m$-sequence} of $\ov{x}$ to be an unordered collection of $n$ paths $\phi_j$ in $X$, with respective non-negative weights $\lambda_j$, such that
\vspace{-2mm}
\begin{enumerate}
        \itemsep-0.25em 
\item the sum of the weights is $1$, i.e., $\sum_{j=1}^n \lambda_j = 1$, and
\item for each $1 \le j \le n$, we have $\phi_j(t_i) = x_i$ for all $1 \le i \le m$.
\end{enumerate} 
Hence, the desired sequential motion planning algorithm is a continuous assignment of each $\ov{x} \in X^m$ to an $n$-distributed $m$-sequence of $\ov{x}$, which is an {\em unordered probability distribution}. For the case $m=2$, our above algorithm is in contrast to Farber's algorithm~\cite[Section 12]{Far2} corresponding to a probabilistic interpretation of $\TC$, whose output is an \textit{ordered} probability distribution. We note that as shown in~\cite[Section 13]{Far2}, the probabilistic definition of $\TC$ from~\cite{Far2} coincides with the original definition of $\TC$ from~\cite{Far1} for simplicial complexes.

The continuous algorithm we seek for the above problem is a natural generalization of the one for $m = 2$ from~\cite{DJ} in the following sense. For a metric space $Z$, let $\mathcal{B}(Z)$ denote the set of probability measures on $Z$ and 
$$
\mathcal{B}_{n}(Z) = \left\{\mu \in \mathcal{B}(Z) \mid |\supp(\mu)| \le n \right\}
$$
denote the space of probability measures on $Z$ supported by at most $n$ points, equipped with the Lévy--Prokhorov metric~\cite{Pr} (see also~\cite[Section 3.1]{DJ}). The Lévy--Prokhorov metric on $\B_n(Z)$ is defined by setting the distance between any two given probability measures $\mu,\nu\in \B_n(Z)$ as
\[
\inf\left\{\epsilon > 0 \hspace{1mm} \middle | \hspace{1mm} \mu(A) < \mu(A^{\epsilon}) + \epsilon \text{ and } \nu(A) < \nu(A^{\epsilon}) + \epsilon \text{ for all } A\in \mathscr{B}(Z)\right\},
\] 
where $\mathscr{B}(Z)$ is the Borel sigma algebra on $Z$ and $A^{\epsilon}$ is the $\epsilon$-neighborhood of $A$ defined as $A^{\epsilon}=\{x\in Z \mid d(x,y) < \epsilon \text{ for some } y\in A\}$. If $\mu \in \mathcal{B}_n(Z)$, then 
$$
\mu=\sum_{z\in F\subset Z,\ |F|\le n}\lambda_z z,
$$
where $\lambda_z\ge 0$ and $\sum\lambda_z=1$, and $\supp(\mu) =\{z\in Z\mid \lambda_z>0\}$. Let $Z = P(X) = \{f \hspace{1mm} | \hspace{1mm} f: [0,1] \to X \text{ is continuous}\}$ be the path space of $X$ with the compact-open topology, and let
$$
P(\ov{x}) = \{f \in P(X) \hspace{1mm} | \hspace{1mm} f(t_i) = x_i \text{ for all } 1 \le i \le m\}
$$
for any $\ov{x} \in X^m$. Then $\mathcal{B}_{n}(P(\ov{x}))$ is the space of all $n$-distributed $m$-sequences of $\ov{x}$. The desired $m$-navigation algorithm for an advanced system with configuration space $X$ is then a continuous map
$$
s_m : X^m \to \mathcal{B}_n(P(X))
$$ 
such that $s_m(\ov{x}) \in \mathcal{B}_{n}(P(\ov{x}))$ for all $\ov{x} \in X^m$. 

We note that the breaking of the system into $n$ pieces while traveling is like an $n$-th degree discontinuity of its motion. Since we are looking for the most optimal algorithm, we want to minimize this number $n$. This gives us the notion of the $m$-th {\em sequential distributional topological complexity} of a space $X$, denoted $\dTC_m(X)$.

\subsection{About this paper}
At this stage, we mention that typically, one takes 
$$
t_i = \frac{i-1}{m-1} \hspace{5mm} \text{for all} \hspace{5mm} 1 \le i \le m,
$$
as in~\cite{Ru},~\cite{LS}. So, in this paper, we will take $t_i = (i-1)/(m-1)$ to get equally-timed motions and we will prove all our results for these $t_i$. Thus, the sequential distributional topological complexity is discussed in this paper {\em only} in this sense. However, we note that with minor modifications, suitable analogs of most of our statements and results from Sections~\ref{Higher Distributional Complexity},~\ref{Cohomological Lower Bounds}, and~\ref{forhspaces} will also hold for any arbitrary $t_i \in [0,1]$, satisfying $t_i < t_{i+1}$ for all $1 \le i \le m-1$, in case of a more general, parametrized motion discussed, for example, in~\cite{CFW} and~\cite{FP}.

In this paper, we are considering the {\em normalized} version of $\TC_m$ in the sense that if $X$ is a contractible space, then $\TC_m(X) = 0$ for all $m \ge 2$. 

The paper is organized as follows. In Section~\ref{Preliminaries}, we set the ground for our paper and recall the notion and characterization of the classical sequential topological complexity, $\TC_m$, for $m \ge 2$. In Section~\ref{Higher Distributional Complexity}, we formally define sequential versions, $\dTC_m$, of distributional topological complexity for $m \ge 2$, prove their homotopy invariance, generalize results for $\dTC$ to $\dTC_m$, and compare these new numerical invariants $\dTC_m(X)$ for a space $X$ with $\TC_m(X)$ and the distributional category ($\dcat$)~\cite{DJ} of $X^{m-1}$ and $X^m$. Section~\ref{Cohomological Lower Bounds} is devoted to finding sharp lower bounds to $\dTC_m(X)$ in Alexander--Spanier cohomology, using the cohomology of the symmetric products of $X^m$, in a way slightly different from the one used to bound $\dTC(X)$ from below in~\cite{DJ}. In Section~\ref{charrr}, we formally define the notion of the distributional sectional category of a fibration, prove its homotopy invariance, and characterize $\dTC_m$ and $\dcat$ from that perspective. In Section~\ref{forhspaces}, we rigorously prove that $\dTC_{m+1}(X)$ and $\dcat(X^m)$ agree for all path-connected CW $H$-spaces $X$ for all $m \ge 1$. Finally, Section~\ref{estcomp} involves explicit computations and estimates of $\dTC_m$ for various classes of closed manifolds, thereby extending some known computations for $\dTC$. 

When the author finished writing this paper, he learned about the article~\cite{KW}, which introduces the notions of \textit{analog} LS-category $(\acat)$ and \textit{analog} sequential topological complexity ($\ATC_m$) of topological spaces with a different motivation and presents some interesting results, particularly for aspherical spaces. Knudsen and Weinberger conjectured in~\cite{KW} that their invariants (defined for compactly generated Hausdorff spaces) coincide with the distributional invariants on metrizable spaces. In this paper, we take a step in that direction and prove in Section~\ref{kwnew} that the distributional invariants give a sharp lower bound to the respective analog invariants on metrizable spaces. In particular, $\dcat(X) \leq \acat(X)$ and $\dTC_m(X) \le \ATC_m(X)$ for all $m \ge 2$. This is especially relevant because, unlike the case of $\dcat$ and $\dTC_m$, no non-trivial lower bounds to $\acat$ and $\ATC_m$ are known in general (even for finite CW complexes) at the time this paper was written. Using our computations from this paper, we also find the sequences $\{\ATC_m(X)\}_{m\ge 2}$ for some closed manifolds $X$. 

\subsection*{Notations and conventions} In this paper, the term \emph{space} refers to a path-connected metric space, \emph{map} refers to a topologically continuous function, and \emph{ring} refers to a commutative ring with unity.

\section{Preliminaries}\label{Preliminaries}
First, we prove an easy result that will be used in the subsequent sections. For any $X$, $m \ge 2$, and $a_i \in (1,\infty)$ such that $a_i > a_{i+1}$ for all $1 \le i \le m-2$, let
$$
T_{m}(X) = \left\{\left(f_{1}, \ldots, f_{m} \right) \in (P(X))^{m} \mid f_{i}(1) = f_{i+1}(0) \text{ for all } 1 \le i \le m-1 \right\}
$$
and $\theta_{m}: T_{m}(X) \to P(X)$ be defined as $\theta_{m}\left(f_{1}, \ldots, f_{m} \right) = f_{1} \star \cdots \star f_{m}$, where
$$
\left(f_{1} \star \cdots \star f_{m}\right)(t) = \begin{cases}
    f_{1}(a_{1}t) & : 0 \le t \le \frac{1}{a_{1}} \\
    f_{2}\left(\frac{a_{2}(a_{1}t-1)}{a_{1}-a_{2}}\right) & : \frac{1}{a_{1}} \le t \le \frac{1}{a_{2}}
    \\
    \hspace{6mm} \vdots & \hspace{8mm} \vdots 
    \\
    f_{m-1}\left(\frac{a_{m-1}(a_{m-2}\hspace{0.5mm} t-1)}{a_{m-2}-a_{m-1}}\right) & : \frac{1}{a_{m-2}} \le t \le \frac{1}{a_{m-1}} 
    \\
    f_{m}\left(\frac{1-a_{m-1}\hspace{0.5mm} t}{1-a_{m-1}}\right) & : \frac{1}{a_{m-1}} \le t \le 1
\end{cases}
$$

\begin{lemma}\label{iscont}
    The map $\theta_{m}$ is continuous for each $m \ge 2$.
\end{lemma}

\begin{proof}
Let us take some $f = (f_{1},\ldots,f_{m}) \in T_{m}(X)$ and any sub-basis neighborhood $U \subset P(X)$ of $\theta_{m}(f)$. By definition of the compact-open topology, $U = (K,W) = \left\{\gamma \in P(X) \mid \gamma(K) \subset W\right\}$ for some compact $K \subset I$ and open $W \subset X$. Define sets $K_{1} = \left\{a_{1}t \mid t \in K \cap [0,1/a_{1}] \right\}$,
$$
K_{m} = \left\{\frac{1-a_{m-1}\hspace{0.5mm} t}{1-a_{m-1}} \hspace{1.5mm} \middle| \hspace{1.5mm} t \in K \cap \left[\frac{1}{a_{m-1}},1\right] \right\},
$$ 
and for all $2 \le i \le m-1$ the sets 
$$
K_{i} = \left\{\frac{a_{i}\left(a_{i-1}\hspace{0.5mm} t-1\right)}{a_{i-1}-a_{i}} \hspace{1.5mm} \middle| \hspace{1.5mm} t \in K \cap \left[\frac{1}{a_{i-1}},\frac{1}{a_{i}}\right] \right\}.
$$
It is easy to see that the sets $K_{i}$ are compact for each $1 \le i \le m$. Therefore,
$$
V = T_{m}(X) \cap \prod_{i=1}^{m}(K_{i},W)
$$
is open in $T_{m}(X)$. For any $(F_{1}, \ldots, F_{m}) \in T_{m}(X)$, by definition, $F_{i} \in (K_{i},W)$ for all $1 \le i \le m$ if and only if $F_{1} \star \cdots \star F_{m} \in (K,W)$. Hence, $V$ is a neighborhood of $f$ and $\theta_{m}(V) \subset U$. Thus, $\theta_m$ is continuous at $f \in T_m(X)$.
\end{proof} 

Let us recall the classical definitions of the Lusternik--Schnirelmann category \cite{Ja},~\cite{CLOT} and the sequential topological complexity~\cite{Far1},~\cite{Ru} of a space.

The \emph{Lusternik--Schnirelmann category} (LS-category) of a space $X$, denoted $\cat(X)$, is the minimal number $n$ such that there is a covering $\{U_i\}$ of $X$ by $n+1$ open sets each of which is contractible in $X$.

For given $m \ge 2$, the $m$-th \emph{sequential topological complexity} of a space $X$, denoted $\TC_m(X)$, is the minimal number $n$ such that there is a covering $\{U_i\}$ of $X^m$ by $n+1$ open sets over each of which there is a map $s_i : U_i \to P(X)$ such that for each $\ov{x} = (x_1,\ldots,x_m) \in U_i \subset X^m$, $s_i(\ov{x})(t_j) = x_j$ for all $1 \le j \le m$. 

We also recall the definition of distributional LS-category~\cite{DJ} of a space. Recall that a space $X$ is called $k$-contractible to a point $x_0 \in X$ if there is a continuous map $H:X\to \mathcal B_k(P(X))$ satisfying $H(x)\in\mathcal B_k(P(x,x_0))$ for all $x \in X$.

The {\em distributional LS-category} of a space $X$, denoted $\dcat(X)$, is the minimal number $n$ such that $X$ is $(n+1)$-contractible to some fixed basepoint $x_0 \in X$.

\subsection{Ganea--Schwarz's approach to $\TC_m$}\label{Preliminaries1}
Let $p: E \to B$ be a fibration. For any $n \ge 1$, the iterated fiberwise join of $n$-copies of $E$ along $p$, denoted $\ast^n_B \hspace{0.5 mm}E$, is defined as the space
$$
\ast^n_B \hspace{0.5 mm}E = \left\{\sum_{i=1}^{n} \lambda_{i} e_{i} \hspace{1mm} \middle| \hspace{1mm} e_{i} \in E, \hspace{0.5mm} \sum_{i=1}^{n}\lambda_{i} = 1, \hspace{0.5mm} \lambda_{i} \geq 0, \hspace{0.5mm} p\left(e_{i}\right)= p\left(e_{j}\right)\right\},
$$
where each element is a formal ordered linear combination of elements such that all terms of the form $0e_i$ are dropped from the sum. Similarly, the iterated fiberwise join of $n$-copies of $p$, denoted $\ast^n_B \hspace{0.5mm}p : \ast^n_B \hspace{0.5 mm}E \to B$, is defined as the fibration
$$
\ast^n_B \hspace{0.5mm}p \left(\sum_{i=1}^{n} \lambda_{i} e_{i}\right) = p\left(e_i \right)
$$ 
for any $i$ with $\lambda_{i} > 0$. Now, we recall the notion of the sectional category, also known as the Schwarz genus, of a fibration.

The \textit{sectional category} of a fibration $p: E \to B$, denoted $\secat(p)$, is the minimal number $n$ such that $B$ can be covered by $n+1$ open sets $W_i$ over each of which there exists a partial section $s_i:W_i\to E$ of $p$.

Let us fix some $m \ge 2$. Given a space $X$, let $\pi_m^X: P(X) \to X^m$ be the fibration defined by $\pi_m^X(\phi) = (\phi(t_1),\phi(t_2), \ldots, \phi(t_m))$. Then by definition, we have that $\TC_m(X) = \secat(\pi_m^X)$, see~\cite[Remark 3.2.5]{Ru}. 

The following statement gives Schwarz's characterization of $\secat$.

\begin{prop}[\protect{\cite[Proposition 8.1]{Ja}}]\label{newaddition}
 For a fibration $p:E\to B$, $\secat(p)\le n$ if and only if the fibration $\ast^{n+1}_B \hspace{0.5mm}p : \ast^{n+1}_B \hspace{0.5 mm}E \to B$ admits a section.
\end{prop}


Let us denote the space $\ast^{n+1}_{X^m} \hspace{1mm} P(X)$ by $\mathcal{G}_{m,n}(X)$ and the fibration $\ast^{n+1}_{X^m} \hspace{1mm} \pi_m^X$ by $\pi_{m,n}^X$. Since $\TC_m(X) = \secat(\pi_m^X)$, we use Proposition~\ref{newaddition} to obtain the Ganea--Schwarz characterization of $m$-th sequential topological complexity as below.

\begin{theorem}\label{original2}
For any $X$, $\TC_m(X) \leq n$ if and only if the fibration
$$
\pi_{m,n}^X: \mathcal{G}_{m,n}(X) \to X^m
$$
admits a section.
\end{theorem}

\subsection{Cohomological lower bounds of $\TC_m$}
Given a space $X$ and ring $R$, the cup-length of the cohomology ring $H^*(X;R)$ is defined as the maximal length $k$ of a non-zero cup product $\alpha_1 \smile\cdots\smile\alpha_k\neq 0$ of cohomology classes $\alpha_i$ of positive degrees. For simplicity, we shall call it the cup-length of $X$ with coefficients in $R$.
Let $\Delta : X \to X^m$ be the diagonal map that induces $\Delta^*: H^*(X^m;R) \to H^*(X;R)$. The elements of $\Ker(\Delta^*)$ are called the $m$-th {\em zero-divisors} of $X$. Then the cup-length of the ideal of the $m$-th zero-divisors of $X$ is a sharp lower bound of $\TC_m(X)$, see~\cite{Ru},~\cite{BGRT}.

We note that this lower bound can be obtained from a more general result of A. S. Schwarz, see~\cite[Theorem 4]{Sch}.

\section{Sequential distributional topological complexity}\label{Higher Distributional Complexity}
Given a metric space $Z$, we equip $\mathcal{B}_n(Z)$ with the Lévy--Prokhorov metric~\cite{Pr} (see Section~\ref{contmot}). We focus on the case $Z = P(X)$ when $X$ is a path-connected metric space. In this case, it is well-known that the compact-open topology on $P(X)$ is metrizable.

We recall that for any $\ov{x} = (x_1, \ldots, x_m) \in X^m$, a {\em $k$-distributed $m$-sequence} of $\ov{x}$ is an unordered collection of $k$ weighted paths $\phi_j$ in $X$ such that for each $1 \le j \le k$, $\phi_j(t_i) = x_i$ for all $1 \le i \le m$, where the weights are non-negative and the sum of the weights is $1$. Here, we have $t_i=(i-1)/(m-1)$ for each $i$.

\begin{defn}
A {\em $k$-distributed $m$-navigation algorithm} on a space $X$ is a map
$$
s_m : X^m \to \mathcal{B}_k(P(X))
$$ 
that satisfies $s_m(\ov{x}) \in \mathcal{B}_{k}(P(\ov{x}))$ for all $\ov{x} \in X^m$, in the notations of Section~\ref{contmot}. 
\end{defn}

\begin{defn}
For a given $m \ge 2$, the {\em $m$-th sequential distributional topological complexity}, or alternatively, the {\em $m$-th higher distributional topological complexity} of a space $X$, denoted $\dTC_m(X)$, is the minimal number $n$ such that $X$ admits an $(n+1)$-distributed $m$-navigation algorithm. 
\end{defn}

Each of the following results generalizes the respective statements for $\dTC$ proved in~\cite[Section 3]{DJ}.

\begin{prop}\label{homoinv}
For each $m \ge 2$, $\dTC_m$ is a homotopy invariant.
\end{prop}

\begin{proof}
Let us fix some $m \ge 2$. Let $f: X \to Y$ be a homotopy domination with a continuous right homotopy inverse $g: Y \to X$. We only need to prove that $\dTC_m(Y) \le \dTC_m(X)$. Since $fg \simeq \mathbbm{1}_{Y}$, there exists a homotopy $h: Y \to P(Y)$ defined as $h(y)= h_y$ for each $y \in Y$, where $h_y(0) = y$ and $h_y(1) = fg(y)$. Let $\dTC_m(X) = n$ and 
$$
s_m : X^m \to \mathcal{B}_{n+1}(P(X))
$$
be an $(n+1)$-distributed $m$-navigation algorithm on $X$. Let $f$ induce a continuous map $f_*: \mathcal{B}_{n+1}(P(X)) \to \mathcal{B}_{n+1}(P(Y))$ due to the functoriality of $\mathcal{B}_{n+1}$. Then the composition
$$
f_* s_m \hspace{1mm} g^m : Y^m \to \mathcal{B}_{n+1}(P(Y))
$$
maps each $\ov{y} \in Y^m$ to an $(n+1)$-distributed $m$-sequence of the ordered $m$-tuple $f^m g^m(\ov{y}) = \{fg(y_1), \ldots, fg(y_m)\}$. Let $s_m g^m(\ov{y}) = \sum \lambda_{\phi} \phi$. For each such path $\phi \in \supp(s_m g^m(\ov{y}))$, let us consider the collection $\{\phi_1, \ldots, \phi_{m-1}\}$, where 
$$
\phi_{i}(s) = \phi \left( \frac{s+i-1}{m-1} \right)
$$
for each $1 \le i \le m-1$. So, $\phi_{i}$ is a path in $X$ from $g(y_i)$ to $g(y_{i+1})$. Using these paths and the homotopy $h$, let us define
$$
\wh{\phi} = \left(h_{y_1} \cdot f\phi_1 \cdot \ov{h}_{y_2}\right) \star  \left(h_{y_2} \cdot f\phi_2 \cdot \ov{h}_{y_3}\right) \star \cdots \star \left(h_{y_{m-1}} \cdot f\phi_{m-1} \cdot \ov{h}_{y_m}\right),
$$ 
where $\cdot$ denotes the usual concatenation of paths and $\star$ denotes concatenation done using the map $\theta_{m-1}$ from Lemma~\ref{iscont} with $a_{i} = 1/t_{i+1} = (m-1)/i$ for all $1 \le i \le m-2$. Then the map $\sigma_m : Y^m \to \mathcal{B}_{n+1}(P(Y))$ defined by
$$
\sigma_m(\ov{y}) = \sum_{\phi \hspace{1mm} \in \hspace{1mm} \supp(s_m g^m(\ov{y}))} \hspace{1mm} \lambda_{\phi}\hspace{0.5mm}\wh{\phi}
$$
is an $(n+1)$-distributed $m$-navigation algorithm on $Y$. Therefore, we obtain the inequality $\dTC_m(Y) \leq n = \dTC_m(X)$.
\end{proof}

The following proposition is straightforward (see~\cite[Proposition 3.10]{DJ} for a proof in the special case $m=2$).

\begin{prop}\label{ineq1}
For each $m \ge 2$, $\dTC_m(X) \leq \TC_m(X)$.
\end{prop}

\begin{prop}\label{ineq2}
For each $m \ge 2$, $\dcat(X^{m-1}) \le \dTC_m(X)$.
\end{prop}

\begin{proof}
Let $\dTC_m(X) = n$. Then there exists an $(n+1)$-distributed $m$-navigation algorithm on $X$, say
$$
s_m: X^m \to \mathcal{B}_{n+1}(P(X)).
$$
For a fixed basepoint $x_0 \in X$, let $J_m : X^{m-1} \hookrightarrow X^m$ be the map 
$$
J(x_1, \ldots, x_{m-1}) = (x_1, \ldots, x_{m-1}, x_0).
$$
Let $\wt{x} = (x_1, \ldots, x_{m-1})$ and $s_m J(\wt{x}) = \sum \lambda_{\phi} \phi$. For each such $\phi\in\supp(s_m J(\wt{x}))$, we can write $\phi' = (\phi_1, \ldots, \phi_{m-1})$, where 
$$
\phi_i(s) = \phi \left(\frac{s(m-i) + i-1}{m-1} \right)
$$
for each $1 \le i \le m-1$. Then $\phi'(0) = \wt{x}$ and $\phi'(1) =\wt{x_0} = (x_0, \ldots, x_0) \in X^{m-1}$. So, the map $\sigma_m : X^{m-1} \to \mathcal{B}_{n+1}(P_0 (X^{m-1}))$ defined by
$$
\sigma_m (\wt{x}) = \sum_{\phi \hspace{1mm} \in \hspace{1mm} \supp(s_m J(\wt{x}))} \hspace{1mm} \lambda_{\phi} \hspace{0.5mm} \phi'
$$
is an $(n+1)$-contraction of $X^{m-1}$ to $\wt{x_0}$. Hence, $\dcat(X^{m-1}) \le n= \dTC_m(X)$.
\end{proof}

\begin{prop}\label{ineq3}
For each $m \ge 2$, $\dTC_m(X) \le \dcat(X^{m})$.
\end{prop}

\begin{proof}
Let $\dcat(X^{m}) = n$ and let us have an $(n+1)$-contraction
$$
H: X^m \to \mathcal{B}_{n+1}(P_0(X^m))
$$
of $X^m$ to its fixed basepoint $(x_1^0, \ldots, x_m^0)$. For some $\ov{x} = (x_1, \ldots, x_m) \in X^{m}$, let $H(\ov{x}) = \sum \lambda_{\phi} \phi$. Since $\phi \in \supp(H(\ov{x})) \subset P_0 (X^m)$, we can consider
$$
\alpha_i = \text{proj}_i \hspace{0.5mm} \phi \in P(X),
$$
so that $\alpha_i(0) = x_i$ and $\alpha_i(1) = x_i^0$ for all $1 \le i \le m$. Here, $\text{proj}_i:X^m\to X$ is the projection map onto the $i$-th factor. Let us fix a path $\gamma_i \in P(X)$ from $x_i^0$ to $x_{i+1}^0$ for each $1 \le i \le m-1$. Define
$$
\wh{\phi} = \left(\alpha_1 \cdot \gamma_1 \cdot \ov{\alpha}_2 \right) \star \left(\alpha_2 \cdot \gamma_2 \cdot \ov{\alpha}_3 \right) \star \cdots \star \left(\alpha_{m-1} \cdot \gamma_{m-1} \cdot \ov{\alpha}_m \right),
$$
where $\star$ is the concatenation via the map $\theta_{m-1}$ with $a_i = 1/t_{i+1} = (m-1)/i$ for all $1 \le i \le m-2$. Then the map $\sigma_m : X^m \to \mathcal{B}_{n+1}(P(X))$ defined by
$$
\sigma_m(\ov{x}) = \sum_{\phi \hspace{1mm} \in \hspace{1mm} \supp(H(\ov{x}))} \hspace{1mm} \lambda_{\phi}\hspace{0.5mm}\wh{\phi}
$$
is an $(n+1)$-distributed $m$-navigation algorithm on $X$. Therefore, we obtain $\dTC_m(X) \leq n = \dcat(X^m)$.
\end{proof}

The following statement justifies that $\{\dTC_m(X)\}_{m\ge 2}$ is a non-decreasing sequence for any fixed space $X$.

\begin{prop}\label{nondec}
For each $m \ge 2$, $\dTC_m(X) \le \dTC_{m+1}(X)$.
\end{prop}

\begin{proof}
The inequality $\dTC_m(X) \le \dcat(X^{m}) \le \dTC_{m+1}(X)$ follows directly from Propositions~\ref{ineq2} and~\ref{ineq3}.
\end{proof}

\begin{prop}\label{ineq4}
For any $m \ge 2$, $\max\{\dTC_m(X), \dTC_m(Y)\} \le \dTC_m(X \times Y)$ for any given spaces $X$ and $Y$.
\end{prop}

\begin{proof}
Let us fix some $b \in Y$ and define $J^m: X^m \to (X \times Y)^m$ by 
$$
J^m\left(x_1, x_2, \ldots, x_m\right) = \left(\left(x_1,b\right), \left(x_2,b\right), \ldots, \left(x_m,b\right)\right).
$$
Let us denote $\ov{x} = (x_1, \ldots, x_m)$. Let $\dTC_m(X \times Y) = n$ with an $(n+1)$-distributed $m$-navigation algorithm
$$
s_m: (X \times Y)^m \to \mathcal{B}_{n+1}(P(X \times Y))
$$
on $X \times Y$. Let $s_m J^m(\ov{x}) = \sum \lambda_{\phi} \phi$. If $\gamma : P(X \times Y) \to P(X)$ is defined as
$$
\gamma (f)= \text{proj}_1 \hspace{0.5mm} f,
$$
where $\text{proj}_1:X\times Y\to X$ denotes the projection onto $X$, then $\gamma$ induces a map $\gamma^*: \mathcal{B}_{n+1}(P(X \times Y)) \to \mathcal{B}_{n+1}(P(X))$ by the functoriality of $\mathcal{B}_{n+1}$. Hence, the map $\sigma_m = \gamma^* s_m J^m: X^m \to \mathcal{B}_{n+1}(P(X))$ defined as
$$
\sigma_m (\ov{x}) = \sum_{\phi  \hspace{1mm} \in \hspace{1mm} \supp(s_m J^m(\ov{x}))} \hspace{1mm} \lambda_{\phi}  \hspace{1mm} \text{proj}_1 \hspace{0.5mm} \phi
$$
is an $(n+1)$-distributed $m$-navigation algorithm on $X$. This gives us the inequality $\dTC_m(X) \le n =  \dTC_m(X \times Y)$. By similar arguments, the other inequality $\dTC_m(Y) \le n =  \dTC_m(X \times Y)$ follows.
\end{proof}

\begin{prop}\label{badupperbound}
For any $m \ge 2$ and space $X$, $\dTC_m(X) \le (\dTC(X)+1)^{m-1} - 1$.
\end{prop}

\begin{proof}
Let $\dTC(X)=\dTC_2(X)=n$. Consider an $(n+1)$-distributed $2$-navigation algorithm on $X$, say
$$
H: X \times X \to \mathcal{B}_{n+1}(P(X)).
$$
Let us fix some $m \ge 3$. Given $\ov{x} = (x_1,\ldots, x_m) \in X^m$, for each $1 \le i \le m-1$, we let
$$
H\left(x_i,x_{i+1}\right) = \sum_{j=1}^{n+1} b_i^j \hspace{0.5mm} \phi_i^j ,
$$
where $\sum_{j=1}^{n+1} b_i^j = 1$ and $\phi_i^j(0) = x_i$ and $\phi_i^j(1) = x_{i+1}$ for each $j$. Consider
$$
\phi_1^{k_1} \star \phi_2^{k_2} \star \cdots \star \phi_{m-1}^{k_{m-1}} \hspace{0.5mm}\in P(X),
$$
where $k_i \in L= \{1,\ldots,n+1\}$ for each $1 \le i \le m-1$ and $\star$ denotes concatenation done via the map $\theta_{m-1}$ from Lemma~\ref{iscont} with $a_j = 1/t_{j+1} = (m-1)/j$ for all $1 \le j \le m-2$. Then the map $s_m : X^m \to \mathcal{B}_{(n+1)^{m-1}}(P(X))$ defined by
$$
s_m(\ov{x}) = \sum_{k_1 \in L}\hspace{1mm}  \sum_{k_2 \in L}  \cdots \sum_{k_{m-1} \in L}  \hspace{1mm} b_1^{k_1} b_2^{k_2} \cdots b_{m-1}^{k_{m-1}} \hspace{0.5mm} \left(\phi_1^{k_1} \star \phi_2^{k_2} \star \cdots \star \phi_{m-1}^{k_{m-1}}\right)
$$ 
is an $(n+1)^{m-1}$-distributed $m$-navigation algorithm on $X$. Hence, we obtain the inequality $\dTC_m(X) \le (n+1)^{m-1}-1=(\dTC(X)+1)^{m-1}-1$.
\end{proof}

\begin{cor}\label{tcmrpn}
For any $m \ge 2$ and $n \ge 1$, $\dTC_m(\R P^n) \le 2^{m-1} - 1$.
\end{cor}

\begin{proof}
This follows directly from Proposition~\ref{badupperbound} and the equality $\dTC(\R P^n)=1$, whose proof can be found in~\cite[Example 3.13]{DJ}.
\end{proof}

For $m = 2$ and $X = \R P^n$, this upper bound is sharp in view of~\cite[Example 3.13]{DJ} for each $n \ge 1$. Using Proposition~\ref{ineq1}, we obtain
$$
\dTC_m (\R P^n) \le \TC_m (\R P^n) \le \cat((\R P^n)^m) = mn.
$$
For a fixed $m \ge 2$, this upper bound is linear in $n$. So our bound from Corollary~\ref{tcmrpn}, which is independent of $n$, is significantly better when $n$ is large. However, using the \textit{analog invariants}, we will improve this upper bound in Section~\ref{seclast}.

\begin{remark}
Let us fix some $m \ge 3$. Due to Corollary~\ref{tcmrpn}, we obtain for all $n \ge (2^{m-1}-1)/(m-2)$ that
$$
\TC_{m}(\R P^n) - \dTC_m(\R P^n) \ge \cat((\R P^n)^{m-1}) - \dTC_m(\R P^n)
$$
$$
\ge n(m-1) - (2^{m-1} - 1) \ge n.
$$
For $m = 2$ and any $n \ge 1$, we get 
$$
\TC(\R P^{n+1}) - \dTC(\R P^{n+1}) \ge \cat(\R P^{n+1}) - \dTC(\R P^{n+1}) \ge n.
$$
Hence, the gap between $\TC_m(X)$ and $\dTC_m(X)$ can be arbitrarily large for each $m \ge 2$.  So, in particular, $\dTC_m$ is a different notion than $\TC_m$ for each $m \ge 2$ and the inequality in Proposition~\ref{ineq1} can be strict for all $m \ge 2$. 
\end{remark}

Since $\dTC_m$ is a homotopy invariant, for a discrete group $\Gamma$, we can define $\dTC_m(\Gamma):= \dTC_m(B\Gamma)$, where $B\Gamma = K(\Gamma, 1)$ is a classifying space of the universal cover of spaces having $\Gamma$ as their fundamental group. In Example~\ref{lateref}, we will see that $\dTC_m(\R P^{\infty}) = \dTC_m(\Z_{2}) =1$. 

\begin{remark}
We note that an analog of~\cite[Theorem 2.1]{FO} does not hold in the case of $\dTC_m$ for any $m \ge 2$, at least for groups with torsion. This is because for $\Gamma = \Z_{2}$, since the subgroup $K \subset \Z_2^m$ (as defined in~\cite[Theorem 2.1]{FO}) is finite, $\cd(K)$ is infinite, where $\cd$ denotes the cohomological dimension~\cite{Br}. But on the other hand, $\dTC_m(\Z_2)=1$ in view of Example~\ref{lateref}.
\end{remark}

\section{Cohomological Lower Bounds}\label{Cohomological Lower Bounds}
\subsection{Symmetric products}\label{4.1}
For a space $X$ and $k \ge 1$, its $k$-th symmetric product $SP^{k}(X)$ is defined as the orbit space of the action of the symmetric group $S_{k}$ on the product space $X^k$ by permutation of coordinates. In this section, we regard each $[x_1, \ldots, x_k] \in SP^k(X)$ as a formal sum $\sum n_i x_i$, where $n_i \ge 1$ and $\sum n_i = k$, subject to the equivalence $n_1 x + n_2 x = (n_1 + n_2)x$. So, $n_i$ denotes the number of times $x_i$ appears in the unordered $k$-tuple $[x_1, \ldots, x_k]$. Define $\delta_k : X \to SP^k(X)$ as the diagonal inclusion $\delta_k (x) = [x, x, \ldots, x] = kx$. 

The following result in singular cohomology will be very useful in Section~\ref{estcomp}.

\begin{prop}[\protect{\cite[Proposition 4.3]{DJ}}]\label{useful}
For a finite simplicial complex $X$ and any $k \ge 1$, the induced homomorphism $\delta_k^* : H^*(SP^k(X);\Q) \to H^*(X;\Q)$ is surjective.
\end{prop}

In this section, we regard $X$ as the subspace of $SP^k(X)$ under the diagonal inclusion $\delta_k$ and use the term {\em inclusion} to refer to the map $\delta_k$.

\subsection{Lower bound for $\dTC_m$}\label{subsect4.2}

\begin{lemma}\label{first}
For any $k \ge 1$, $SP^k(P(X))$ deforms to $SP^k(X)$.
\end{lemma}

\begin{proof}
Let $\text{ev}: P(X) \to X$ be defined as the evaluation fibration $\text{ev}: \phi \mapsto \phi(0)$. Let $f = SP^k(\text{ev}) : SP^k(P(X)) \to SP^k(X)$ be induced by $\text{ev}$ due to the functoriality of the $k$-th symmetric product $SP^k$. Define $g: SP^k(X) \to SP^k(P(X))$ as
$$
g[a_1, \ldots, a_k] = \left[c_{a_1}, \ldots, c_{a_k}\right],
$$
where $c_{a_i}$ denotes the constant path at $a_i \in X$. Clearly, $fg = \mathbbm{1}_{SP^k(X)}$. Finally, define a map $h: SP^k(P(X)) \times I \to SP^k(X)$ such that
$$
h\left(\left[\phi_1, \ldots, \phi_k \right],t\right)(s) = \left[ \phi_1(s(1-t)), \ldots, \phi_k(s(1-t)) \right]
$$
for all $s \in [0,1]$. Then $h$ is a homotopy between $\mathbbm{1}_{SP^k(P(X))}$ and $gf$.
\end{proof}

For fixed $m \ge 2$ and space $X$, the fibration $\pi_m: P(X) \to X^m$ defined in Section~\ref{Preliminaries1} induces by functoriality $\zeta_n = SP^{n!}(\pi_m): SP^{n!}(P(X)) \to SP^{n!}(X^m)$ for each $n \ge 1$. Let $\pa_n: X^m \to SP^{n!}(X^m)$ be the diagonal inclusion in the above sense. Consider the following pullback diagram.
\begin{equation}\label{ones}
\begin{tikzcd}[contains/.style = {draw=none,"\in" description,sloped}]
\mathcal{D}_{n,m} \arrow{r}{a} \arrow[swap]{d}{\sigma_n} 
& 
SP^{n!}(P(X)) \arrow[swap]{d}{\zeta_n}
\\
X^m  \arrow{r}{\pa_n}
& 
SP^{n!}(X^m).
\end{tikzcd}
\end{equation}
Here, for $t_j = (j-1)/(m-1)$, we have
$$
\mathcal{D}_{n,m} = \left\{ \left( \left(x_1, \ldots, x_m\right), \left[\phi_1, \ldots, \phi_{n!}\right] \right) \hspace{1mm} \middle| \hspace{1mm} \phi_i(t_j) = x_j \text{ for all } 1 \le j \le m \right\},
$$ 
and $\sigma_n$ is the canonical projection, the pullback of $\zeta_n$ along $\pa_n$. 

We note that in this setting, one can replace the set $X^m$ on the left with any subset $A_i\subset X^m$ and obtain a pullback diagram as above with the corresponding pullback space $\mathcal{D}_{n,m}^i$, the top horizontal map $a_i$, and the canonical projection fibration $\sigma_n^i:\mathcal{D}_{n,m}^i \to A_i$ which is the pullback of $\zeta_n$ along the restriction of the map $\pa_n$ to $A_i$. For simplicitly, we shall write this restriction as $\pa_n:A_i\to SP^{n!}(X^m)$.

\begin{lemma}\label{second}
If $\dTC_m(X) < n$, then there exist sets $A_1, A_2, \ldots, A_n$ that cover $X^m$ and over each of which $\sigma_n$ has a section.
\end{lemma}

\begin{proof}
Since $\dTC_m(X) < n$, there exists an $n$-distributed $m$-navigation algorithm
$$
s_m: X^m \to \mathcal{B}_{n}(P(X))
$$
on $X^m$. For each $1 \le i \le n$, define 
$$
A_i = \left\{ \ov{x} = (x_1, \ldots, x_m) \in X^m \mid |\supp (s_m(\ov{x}))| = i\right\}.
$$
We define a map $H_i : A_i \to SP^{n!}(P(X))$ as
$$
H_i (\ov{x}) = \sum_{\phi \hspace{1mm} \in \hspace{1mm} \supp (s_m(\ov{x}))} \hspace{1mm} \frac{n!}{i} \phi.
$$
For each $\ov{x} \in X^m$, see that $\zeta_n H_i(\ov{x}) = n! \hspace{0.5mm} \ov{x}\in SP^{n!}(X^m)$. Hence, $\zeta_n H_i = \pa_n$. So, the following diagram commutes.
\begin{equation}\label{twos}
\begin{tikzcd}
A_i \arrow[ddr,bend right,"\mathbbm{1}_{A_i}"'] \arrow[drr,bend left,"H_i"] \arrow[dr,dashed,"\tau_i"] 
\\
& 
\smash[b]{\mathcal{D}_{n,m}^i} \arrow{r}{a_i} \arrow[d,"\sigma_n^i"] 
& 
SP^{n!}(P(X)) \arrow[d,"\zeta_{n}"] 
\\
& 
A_i \arrow[r,"\pa_n"] 
& 
SP^{n!}(X^m).
\end{tikzcd}
\end{equation}
Here, $\tau_i : A_i \to \mathcal{D}_{n,m}$ exists due to the universal property of pullback and we have $\sigma_n^i \tau_i = \mathbbm{1}_{A_i}$. So, we have found a section of $\sigma_n^i$ over each $A_i$.
\end{proof}

For a fixed $m \ge 2$, let $\Delta: X \to X^m$ be the diagonal map and let the map $\Delta_n : SP^{n!}(X) \to SP^{n!}(X^m)$ be induced from $\Delta$ by functoriality of $SP^{n!}$. We use an idea from~\cite[Theorem 7]{Far1} and~\cite[Section 2]{Sh} to prove the following result in Alexander--Spanier cohomology~\cite{Sp}.

\begin{thm}\label{boundtcm}
Suppose that $\alpha_i^* \in H^{k_i}(SP^{n!}(X^m); R)$, $1 \le i \le n$, for some ring $R$ and $k_i \ge 1$, are cohomology classes such that $\Delta_n^*(\alpha_i^*) = 0$. Let $\alpha_i$ be their images under the induced homomorphism $\pa_n^*$, and let us assume that the cup product $\alpha_1 \smile \alpha_2 \smile \cdots \smile \alpha_n$ is non-zero. Then, 
$$
\dTC_m(X) \ge n.
$$
\end{thm}

\begin{proof}
Let us consider the following commutative diagram, where the continuous map $g:SP^{n!}(X)\to SP^{n!}(P(X))$ is the same as in the proof of Lemma~\ref{first}.
\begin{equation}
\begin{tikzcd}[contains/.style = {draw=none,"\in" description,sloped}]
H^*(SP^{n!}(P(X));R) \arrow{d}{g^*}
& 
\\
H^*(SP^{n!}(X);R)
& 
H^*(SP^{n!}(X^m);R). \arrow{l}{\Delta_n^*} \arrow[swap]{ul}{\zeta_n^*}
\end{tikzcd}
\end{equation}
Due to Lemma~\ref{first}, $g^*$ is an isomorphism. So, in particular, $\Ker(\Delta_n^*)  = \Ker(\zeta_n^*)$. Suppose that $\dTC_m(X) < n$. Then from Lemma~\ref{second}, there exists a cover $\{A_i\}_{i=1}^n$ of $X^m$ such that $\sigma_n^i$ has a section $\tau_i$ over each $A_i$, i.e., $\sigma_n^i \tau_i = \mathbbm{1}_{A_{i}}$. Due to this and an analog of Diagram~\ref{ones} with the roles of $X^m$, $\mathcal{D}_{n,m}$, and $\sigma_n$ replaced, respectively, with $A_i$, $\mathcal{D}_{n,m}^i$, and $\sigma_n^i$, the following diagram commutes for each $1 \le i \le n$.
\begin{equation}\label{threes}
\begin{tikzcd}[every arrow/.append style={shift left}]
H^{k_i}(A_i;R)  \arrow{r}{(\sigma_n^i)^*}
& 
H^{k_i}(\mathcal{D}_{n,m}^i;R) \arrow{l}{{\tau_i^* \vphantom{1}}}
\\
H^{k_i}(SP^{n!}(X^m);R) \arrow{r}{\zeta_n^*} \arrow[swap]{u}{\pa_n^*} 
& 
H^{k_i}(SP^{n!}(P(X));R).  \arrow[swap]{u}{a_i^*}
\end{tikzcd}
\end{equation}
Clearly, $\tau_i^*( \sigma_n^i)^* = \mathbbm{1}_{H^{k_i}(A_i;R)}$. Therefore, $(\sigma_n^i)^*$ is injective. Since $\Delta_n^*(\alpha_i^*) = 0$, we get $\zeta_n^*(\alpha_i^*) = 0$, and thus, 
$$
(\sigma_n^i)^*(\pa_n^*(\alpha_i^*)) = a_i^* (\zeta_n^*(\alpha_i^*)) =0.
$$
Now, because $(\sigma_n^i)^*$ is injective, $\pa_n^*(\alpha_i^*)=0$. From the long exact sequence of the pair $(SP^{n!}(X^m),A_i)$ in Alexander--Spanier cohomology, 
$$
\cdots \to H^{k_i}(SP^{n!}(X^m),A_i; R) \xrightarrow{j_i^*} H^{k_i}(SP^{n!}(X^m);R) \xrightarrow{\pa_n^*} H^{k_i}(A_i;R) \to \cdots,
$$
there exists $\ov{\alpha_i^*} \in H^{k_i}(SP^{n!}(X^m),A_i; R)$ such that $j_i^*(\ov{\alpha_i^*}) = \alpha_i^*$. Further, let $\pa_n^*(\ov{\alpha_i^*}) = \ov{\alpha}_i \in H^{k_i}(X^m,A_i; R)$. For $j$ and $j'$ denoting the sums of maps $j_i$ and $j_i'$, respectively, and $k = \sum k_i$, we get the following commutative diagram.
\begin{equation}
\begin{tikzcd}[contains/.style = {draw=none,"\in" description,sloped}]
H^{k}(X^m;R)  
& 
\arrow[swap]{l}{(j')^*} H^{k}\left(X^m,\bigcup_{i=1}^{n}A_i;R\right)  
\\
H^{k}\left(SP^{n!}(X^m);R\right)  \arrow[swap]{u}{\pa_n^{*}}
& 
H^k \left(SP^{n!}(X^m),\bigcup_{i=1}^n A_{i};R\right). \arrow[swap]{u}{\pa_n^{*}} \arrow[swap]{l}{j^*} 
\end{tikzcd}
\end{equation}
The cup product $\ov{\alpha_1^*} \smile \cdots \smile \ov{\alpha_n^*}$ in the bottom-right goes to the non-zero cup product $\alpha_1 \smile \cdots \smile \alpha_n \neq 0$ in the top-left. But in the process, it factors through $\ov{\alpha}_1 \smile \cdots \smile \ov{\alpha}_n \in H^k (X^m,X^m;R) = 0$. This is a contradiction. Hence, we must have $\dTC_m(X) \ge n$.
\end{proof}

We note that the reason we work in Alexander--Spanier cohomology in Theorem~\ref{boundtcm} is that the sets $A_i$ defined in Lemma~\ref{second} need not be open or closed in $SP^{n!}(X^m)$. So, $(SP^{n!}(X^m),A_i)$ may not be a good pair in the sense of~\cite{Ha}.

\begin{remark}
If $X$ is a locally finite CW complex, then Theorem~\ref{boundtcm} for $m=2$ is recovered from~\cite[Theorem 4.12]{DJ} due to the fact that the Alexander--Spanier cohomology groups coincide with the singular cohomology groups for locally finite CW complexes~\cite{Sp}, and the following commutative diagram.
\begin{equation}
\begin{tikzcd}[contains/.style = {draw=none,"\in" description,sloped}]
H^{*}(SP^{n!}(X^2);R)  \arrow{r}{\Delta_n^*} \arrow[swap]{d}{\pa_n^*}
& 
H^{*}\left(SP^{n!}(X);R\right)  \arrow{d}{}  
\\
H^{*}\left(X^2;R\right)  \arrow{r}{\Delta^*} 
& 
H^* \left(X;R\right).
\end{tikzcd}
\end{equation}
Hence, $\Delta_n^*(\alpha_i^*) = 0$ implies that $\Delta^*(\alpha_i) = 0$, i.e., $\alpha_i$ is a $2^{nd}$ zero-divisor. 
\end{remark}

\section{Distributional Sectional Category}\label{charrr}

Let $p : E\to B$ be a Hurewicz fibration. For each $n \ge 1$, define a space
$$
E_n(p) = \bigcup_{x \in B}\mathcal{B}_n  (p^{-1}(x)) = \left\{ \mu \in B_n(E) \hspace{1mm} \middle| \hspace{1mm} \supp(\mu) \subset p^{-1}(x), \hspace{1mm} x \in B \right\}
$$
and a map $\mathcal{B}_n(p) : E_n(p) \to B$ such that
$$
\mathcal{B}_n(p)(\mu) = x \hspace{4mm} \text{whenever} \hspace{4mm} \mu \in \mathcal{B}_n (p^{-1}(x)).
$$

\begin{prop}[\protect{\cite[Proposition 5.1]{DJ}}]\label{fibration}
The mapping $\mathcal{B}_n(p) : E_n(p) \to B$ is a Hurewicz fibration.
\end{prop}

\begin{defn}
Given a fibration $p: E \to B$, its \textit{distributional Schwarz genus}, or alternatively, its \textit{distributional sectional category}, denoted $\dsecat(p)$, is the minimal number $n$ such that $\mathcal{B}_{n+1}(p) : E_{n+1}(p) \to B$ admits a section.
\end{defn}

Note that this definition is a distributional analog of the equivalent definition of $\secat$ as it follows from Proposition~\ref{newaddition}. From Proposition~\ref{newaddition}, it is also to see that $\dsecat(p)\le \secat(p)$ for each fibration $p:E\to B$.

We also note that the notion of analog sectional category was introduced and studied in~\cite{KW} which is closely related with the notion of $\dsecat$.

\begin{prop}\label{hinvsecat}
    $\dsecat$ is a homotopy invariant.
\end{prop}

\begin{proof}
    Given fibrations $p: E \to B$ and $q: Z \to C$ and the following commutative diagram where the horizontal maps are homotopy equivalences,
\begin{equation}
    \begin{tikzcd}[every arrow/.append style={shift left}]
E \arrow{r}{f} \arrow[swap]{d}{p}
& 
Z  \arrow{l}{{f' \vphantom{1}}}  \arrow{d}{q}  
\\
B \arrow{r}{g}
& 
C. \arrow{l}{{g' \vphantom{1}}} 
\end{tikzcd}
\end{equation}
 we need to show that $\dsecat(p) = \dsecat(q)$. For any $k \ge 1$, we first obtain fibrations $\mathcal{B}_k(p): E_k(p) \to B$ and $\mathcal{B}_k(q): Z_k(q) \to C$ by Proposition~\ref{fibration}. The functoriality of $\mathcal{B}_k$ gives a map $f_* : \mathcal{B}_k(E) \to \mathcal{B}_k(Z)$ such that
 $$
 f_*\left(\sum \lambda_{r} r\right) = \sum \lambda_{r} \hspace{0.5mm} f(r).
 $$
 Let $\mu = \sum \lambda_{r} r \in \mathcal{B}_k(p^{-1}(x)) \subset E_k(p)$ for some $x \in B$. So, $\text{supp}(\mu) \subset p^{-1}(x)$, which means $p(r) = x$ whenever $\lambda_r > 0$. Note that $q(f(r)) = g(p(r)) = g(x)$ whenever $\lambda_r > 0$. So, $\text{supp}(f_*(\mu)) \subset q^{-1}(g(x))$ and thus, $f_*(\mu)\in\mathcal B_k(q^{-1}(g(x))) \subset Z_k(q)$. Let $F$ denote the restriction of $f_*$ to $E_k(p)$. Thus, $F: E_k(p) \to Z_k(q)$ is defined. Similarly, $F': Z_k(q) \to E_k(p)$ is obtained as the restriction of $f'_*$ to $Z_k(q)$. If $\mu \in \mathcal{B}_k(p^{-1}(x)) \subset E_k(p)$ for some $x \in B$, then
 $$
 \left(\mathcal{B}_k(q) \hspace{0.5mm} F \right)(\mu) = g(x) = \left(g \hspace{0.5mm} \mathcal{B}_k(p)\right)(\mu).
 $$
Similarly, if $\vartheta \in B_k(q^{-1}(y)) \subset Z_k(q)$ for some $y \in C$, then
 $$
 \left(\mathcal{B}_k(p) \hspace{0.5mm} F'\right) (\vartheta) = g'(y) = \left(g' \hspace{0.5mm} \mathcal{B}_k(q)\right)(\vartheta).
 $$
So we get the following commutative diagram.
 \begin{equation}
    \begin{tikzcd}[every arrow/.append style={shift left}]
E_k(p) \arrow{r}{F} \arrow[swap]{d}{\mathcal{B}_k(p)}
& 
Z_k(q) \arrow{l}{{F' \vphantom{1}}}  \arrow{d}{\mathcal{B}_k(q)} 
\\
B \arrow{r}{g}
& 
C. \arrow{l}{{g' \vphantom{1}}} 
\end{tikzcd}
\end{equation}
Let $\dsecat(p) = n -1$. Then by definition, there exists a section $s: B \to E_n(p)$ of $\mathcal{B}_n(p)$. Consider $F s g' : C \to Z_n(q)$. See that
$$
\mathcal{B}_n(q) F s g' = g \hspace{0.5mm} \mathcal{B}_n(p) s g' = gg' \simeq \mathbbm{1}_{C}.
$$
So, $Fsg'$ is a homotopy section of $\mathcal{B}_n(q)$. Since $\mathcal{B}_n(q)$ is a fibration, the homotopy section $Fsg'$ gives a section of $\mathcal{B}_n(q)$. So, $\dsecat(q) \le n-1 = \dsecat(p)$. Similarly, if $\dsecat(q) = m-1$ and $s': C \to Z_m(q)$ is a section of $\mathcal{B}_m(q)$, then $F's'g$ helps obtain a section of $\mathcal{B}_m(p)$. This gives $\dsecat(p) \le m-1 = \dsecat(q)$.
\end{proof}

Let $h: A \to B$ be any continuous map. It has a fibrational substitute, say $p_{h}: E \to B$. Since any two fibrational substitutes of a given map are fiberwise homotopy equivalent, using Proposition~\ref{hinvsecat}, we can define
$$
\dsecat(h): = \dsecat\left(p_h\right).
$$

\begin{remark}
    We can also define the distributional sectional category of group homomorphisms between discrete groups. Let $\phi: \Gamma \to \Lambda$ be such a homomorphism. This gives a map $B\phi: B \Gamma \to B\Lambda$ that induces $\phi$ at the level of fundamental groups. We define $\dsecat(\phi): = \dsecat(B\phi).$
\end{remark}

\subsection{For $\dTC_m$}\label{chardtcm}
For $m\ge 2$, upon taking $p = \pi_m^X, E = P(X),$ and $B = X^m$, and observing that an $(n+1)$-distributed $m$-navigation algorithm on $X$ produces a section of the fibration $\mathcal{B}_{n+1}(\pi_{m}^X)$ and vice-versa, we get the following Ganea--Schwarz-type characterization of $\dTC_m$.

\begin{prop}\label{chartcm}
For $m\ge 2$, $\dTC_m(X) \le n$ if and only if the fibration
$$
\mathcal{B}_{n+1}(\pi_m^X) : P(X)_{n+1}(\pi_m^X) \to X^m
$$
admits a section.
\end{prop}

Therefore, for any space $X$ and $m \ge 2$, we conclude that
$$
\dTC_m(X) = \dsecat\left(\pi_m^X\right).
$$
\begin{remark}
    This realization provides an alternate proof of the homotopy invariance of $\dTC_m$ (see Proposition~\ref{homoinv}) in light of Proposition~\ref{hinvsecat}. 
\end{remark}

If $\Delta_{m}^X: X \to X^m$ denotes the diagonal map, then since $\pi_m^X$ is the fibrational substitute of $\Delta_m^X$, we have $\dTC_m(X) = \dsecat(\Delta_m^X)$.

\subsection{For $\dcat$}\label{fornotations}
A Ganea--Schwarz-type characterization of $\dcat(X)$ was provided in~\cite[Proposition 5.3]{DJ}. For any given $m \ge 1$, we now obtain a slightly different characterization of $\dcat(X^m)$ as follows. Let $\xi_m^X : P(X) \to X^m$ be the fibration
$$
\xi_m^X (\phi)= \left(\phi\left(\frac{1}{m}\right), \phi\left(\frac{2}{m}\right), \ldots, \phi\left(\frac{m-1}{m}\right), \phi\left(1\right) \right).
$$
From this, we can form the fibration $\mathcal{B}_{n+1}(\xi_m^X): P(X)_{n+1}(\xi_m^X) \to X^m$. For some fixed basepoint $x_0 \in X$, let
$$
\mathcal{P}_{n+1}(X) = \left\{\sum \lambda_{\psi} \hspace{0.5mm} \psi \in \mathcal{B}_{n+1}(P(X)) \hspace{1mm} \middle| \hspace{1mm} \psi(0) = x_0 \text{ when } \lambda_\psi>0\right\}.
$$
If $\mathcal{A}_{n+1}(X) = P(X)_{n+1}(\xi_m^X) \cap \mathcal{P}_{n+1}(X)$, then we let the map $\mathcal{C}_{n+1}(\xi_m^X)$ denote the restriction of the fibration $\mathcal{B}_{n+1}(\xi_m^X)$ to the subspace $\mathcal{A}_{n+1}(X)$. The following statement generalizes~\cite[Proposition 5.3]{DJ}.

\begin{prop}\label{catnew}
For $m\ge 1$, $\dcat(X^m) \le n$ if and only if the map
$$
\mathcal{C}_{n+1}\left(\xi_m^X \right): \mathcal{A}_{n+1}(X) \to X^m
$$
admits a section.
\end{prop}

\begin{proof}
Let $\dcat(X^m) \le n$. Then for a fixed basepoint $\ov{x}_0 = (x_1^0, \ldots, x_m^0) \in X^m$, there exists an $(n+1)$-contraction 
$$
H: X^m \to \mathcal{B}_{n+1}(P_0(X^m))
$$
of $X^m$ to $\ov{x}_0$. For $\ov{x} = (x_1, \ldots, x_m) \in X^m$, let $H(\ov{x}) = \sum \lambda_{\phi} \phi$. Let 
$$
\beta_i = \text{proj}_i \hspace{0.3mm} \ov{\phi} \in P(X),
$$
where $\ov{\phi}(t) = \phi(1-t)$ and $\text{proj}_i:X^m\to X$ is the projection onto the $i$-th factor. Let $\gamma_i \in P(X)$ be a path from $x_i^0$ to $x_{i+1}^0$ for each $i$. Define 
$$
\widehat{\phi} = \beta_1 \star \left(\ov{\beta}_1 \cdot \gamma_1 \cdot \beta_2\right) \star \left(\ov{\beta}_2 \cdot \gamma_2 \cdot \beta_3 \right) \star \cdots \star \left(\ov{\beta}_{m-1} \cdot \gamma_{m-1} \cdot \beta_m\right),
$$
where $\star$ is the concatenation via the map $\theta_m$ from Lemma~\ref{iscont} with $a_i = i/m$ for all $1 \le i \le m-1$. Note that $\widehat{\phi}(0) = x_1^0$, which we shall regard as the basepoint of $X$. Finally, define $K : X^m \to \mathcal{B}_{n+1}(P(X))$ as
$$
K(\ov{x}) = \sum_{\phi \hspace{1mm}  \in \hspace{1mm} \supp(H(\ov{x}))} \hspace{1mm} \lambda_{\phi}\hspace{0.5mm} \widehat{\phi}.
$$
By definition, the range of $K$ is contained in $\mathcal{A}_{n+1}(X)$. Clearly, we have that $\mathcal{C}_{n+1}\left(\xi_m^X\right) K = \mathbbm{1}_{X^m}$. Therefore, $K$ is a section of $\mathcal{C}_{n+1}(\xi_m^X)$. For the converse, let 
$$
\psi : X^m \to \mathcal{A}_{n+1}(X)
$$ 
be a section of $\mathcal{C}_{n+1}(\xi_m^X)$, with a chosen basepoint $x_0 \in X$. Let $\psi(\ov{x}) = \sum \lambda_{\phi} \phi$. For each such path $\phi \in \supp(\psi(\ov{x}))$, we can define $\phi' = (\phi_1, \phi_2, \ldots, \phi_{m})$, where 
$$
\phi_i(s) = \phi \left(\frac{i(1-s)}{m} \right)
$$
for each $1 \le i \le m$. Then $\phi'(0) = \ov{x}$ and $\phi'(1) = (x_0, \ldots, x_0) \in X^{m}$. So, the map $\sigma_m : X^{m} \to \mathcal{B}_{n+1}(P_0 (X^{m}))$ defined by
$$
\sigma_m (\ov{x}) = \sum_{\phi \hspace{1mm} \in \hspace{1mm} \supp(\psi(\ov{x}))} \hspace{1mm} \lambda_{\phi} \hspace{0.5mm} \phi'
$$
is an $(n+1)$-contraction of $X^{m}$ to $(x_0, \ldots, x_0)$. Hence, $\dcat(X^{m}) \le n$.
\end{proof}

Therefore, for any space $X$ and $m \ge 1$, it follows that 
$$
\dcat(X^m) = \dsecat\left(\xi_m^X\right).
$$

For a fixed $x_0 \in X$, let $P_0(X)$ denote the space of paths that end at $x_0$, and let $p_0^X:P_0(X) \to X$ be the fibration that evaluates each path at $t = 0$. Then, in particular for $m = 1$, we have $\dcat(X) = \dsecat(\xi_1^X) = \dsecat(p_0^X).$
\begin{remark}
    This gives an alternate proof of the homotopy invariance of $\dcat$ obtained in~\cite[Proposition 3.2]{DJ}.
\end{remark}

Also, for the inclusion map $\iota_X: x_0 \hookrightarrow X$, we have $\dcat(X) = \dsecat(\iota_X).$ This is in analogy with the well-known fact that $\cat(X)=\secat(\iota_X)$, see~\cite[Page 342]{Ja}.

\subsection{A lower bound for $\dsecat$}
Let $p: E \to B$ be a map and for a fixed $n \ge 1$ let $p_n: SP^{n!}(E) \to SP^{n!}(B)$ be induced by the functor $SP^{n!}$. As in Section~\ref{4.1}, we let $\delta_n: B \to SP^{n!}(B)$ be the diagonal inclusion. In light of Schwarz's cohomological lower bound~\cite{Sch} for $\secat(p)$, and motivated by the analogy between the definitions of $\dsecat(p)$ and $\secat(p)$, we propose the following generalization of Theorem~\ref{boundtcm} and~\cite[Theorem 4.7]{DJ}.

\begin{conjec}
    Suppose $\alpha_i^* \in H^{k_i}(SP^{n!}(B);R)$, $1 \le i \le n$, for some ring $R$ and $k_i \ge 1$, are cohomology classes such that $p_n^*(\alpha_i^*) = 0$. If $\alpha_i$ are their images under the induced homomorphism $\delta_n^*$ such that $\alpha_1 \smile \cdots \smile \alpha_n \neq 0$, then $\dsecat(p) \ge n$.
\end{conjec}

\section{For $H$-spaces}\label{forhspaces}
\begin{defn}
 A CW complex $X$ is called an \textit{$H$-space} if there exists a basepoint $e \in X$, called an {\em identity of} $X$, and a map $\nu: X \times X \to X$, called a {\em product} on $X$, such that the maps $\nu(e,\cdot)$ and $\nu(\cdot,e)$ are homotopic to the identity.
\end{defn}

\begin{remark}
For an $H$-space $X$, it is easy to see that the maps $\nu(e,\cdot)$ and $\nu(\cdot,e)$ are homotopic to the identity if and only if $\nu(x,e)=x=\nu(e,x)$ for all $x\in X$, see, for example,~\cite[Page 1831]{LS}.
\end{remark}

For classical sequential topological complexity and LS-category, when a CW complex $X$ is an $H$-space, then $\TC_{m+1}(X) = \cat(X^m)$ holds for all $m \ge 1$~\cite{LS}. For $\dTC$ and $\dcat$, when $X$ is a topological group, then $\dTC(X) = \dcat(X)$,~\cite{DJ}. 

Here, using ideas from~\cite[Theorem 1]{LS}, we obtain the following generalization.

\begin{theorem}\label{hspaces}
    If $X$ is an $H$-space, then $\dTC_{m+1}(X) = \dcat(X^m)$ for all $m \ge 1$.
\end{theorem}

\begin{proof}
Since $X$ is a CW $H$-space, we have an identity element $e \in X$ and a product $\nu:X \times X \to X$ as above. Following the proof of~\cite[Theorem 1]{LS}, we obtain a map $D : X \times X \to X$ such that $\text{proj}_1 \bullet D \simeq \text{proj}_2$, where $\bullet$ is the product in $[X \times X, X]$ induced by $\nu$ and $\text{proj}_i:X \times X\to X$ are the projections onto the $i$-th factor. Here, the product $\bullet$ on $[X\times X,X]$ is such that any equation $a\bullet z = b$ for $a,b\in [X\times X,X]$ has a unique solution $z\in [X\times X,X]$. We define a map $f_m: X^{m+1} \to X^m$ as 
$$
f_m \left(x,x_1,x_2, \ldots, x_m \right) = \left( D\left(x,x_1\right),D\left(x,x_2\right), \ldots, D\left(x,x_m\right)\right).
$$
Let $\dcat(X^m) = n$. In the notations of Section~\ref{fornotations}, consider the following pullback.
\begin{equation}\label{fours}
\begin{tikzcd}[contains/.style = {draw=none,"\in" description,sloped}]
Q_m  \arrow{r} \arrow[swap]{d}{q_m}
& 
\mathcal{A}_{n+1}(X) \arrow{d}{\mathcal{C}_{n+1}(\xi_m^X)}  
\\
X^{m+1} \arrow{r}{f_m} 
& 
X^m.
\end{tikzcd}
\end{equation}
Here, we regard $e \in X$ as the basepoint of $X$ to form $\mathcal{P}_{n+1}(X)$, and thus $\mathcal{A}_{n+1}(X)$. By definition, the subset $Q_m \subset X^{m+1} \times \mathcal{A}_{n+1}(X)$ is
$$
Q_m = \left\{\left( \left(x, x_1, \ldots, x_m\right), \sum \lambda_{\phi} \phi \right) \hspace{1mm} \middle| \hspace{1mm} \phi(0) = e, \phi\left(\frac{i}{m}\right) = D\left(x,x_i \right) \text{ for all } 1 \le i \le m\right\},
$$
and $q_m$ is the projection onto $X^{m+1}$. For each $y \in X$ and $\psi\in P(X)$, we define a path $\nu(y,\psi)\in P(X)$ such that $\nu(y,\psi)(t) = \nu(y,\psi(t))$ for all $t\in I$. Let us define a map $\Theta_m : Q_m \to \mathcal{B}_{n+1}(P(X))$ as
$$
\Theta_m \left( \left(x, x_1, \ldots, x_m\right), \sum \lambda_{\phi} \phi \right) = \sum \lambda_{\phi} \hspace{0.5mm} \nu(x,\phi).
$$
Thus, applying the fibration $\mathcal{B}_{n+1}(\pi_{m+1})$ from Section~\ref{chardtcm}, we get the following:
$$
\mathcal{B}_{n+1}(\pi_{m+1}) \hspace{1mm} \Theta_m \left( \left(x, x_1, \ldots, x_m\right), \sum \lambda_{\phi} \phi \right) = \left( x, \nu\left(x,D\left(x,x_1\right)\right), \ldots, \nu\left(x,D\left(x,x_m \right)\right) \right).
$$
Now, as in~\cite[Page 5]{LS}, let us define maps $p_i : X^{m+1} \to X$ as $p_i(a_1, \ldots, a_{m+1}) = a_i$, and $p_{1,i}: X^{m+1} \to X^2$ as $p_{1,i}(a_1, \ldots, a_{m+1}) = (a_1,a_{i+1})$. Then 
$$
\nu(x,D(x,x_i)) = (\text{proj}_1 \bullet D) \hspace{1mm} p_{1,i} (x,x_1, \ldots, x_m)
$$
and $\text{proj}_2 \hspace{1mm} p_{1,i} = p_{i+1}$. Since $\text{proj}_1 \bullet D \simeq \text{proj}_2$, we get
$$
\mathcal{B}_{n+1}(\pi_{m+1}) \hspace{1mm} \Theta_m = \left( p_1, (\text{proj}_1 \bullet D)\hspace{1mm}p_{1,1}, \ldots,  (\text{proj}_1 \bullet D)\hspace{1mm}p_{1,m} \right) q_m
$$
$$
\simeq \left( p_1, \text{proj}_2 \hspace{1mm} p_{1,1}, \ldots, \text{proj}_2 \hspace{1mm} p_{1,m} \right)\hspace{0.4mm} q_m = (p_1,p_2, \ldots, p_{m+1}) \hspace{1mm} q_m = q_m.
$$ 
So, $\mathcal{B}_{n+1}(\pi_{m+1}) \Theta_m \simeq q_m$. Since $\dcat(X^m) = n$, there exists a section of $\mathcal{C}_{n+1}(\xi_m^X)$, say $H: X^m \to \mathcal{A}_{n+1}(X)$, due to Proposition~\ref{catnew}. Then we get the following commutative diagram.
\begin{equation}\label{fives}
\begin{tikzcd}
X^{m+1} \arrow[ddr,bend right,"\mathbbm{1}_{X^{m+1}}"'] \arrow[drr,bend left,"H f_m"] \arrow[dr,dashed,"s"] 
\\
& 
\smash[b]{Q_m} \arrow{r} \arrow[d,"q_m"] 
& 
\mathcal{A}_{n+1}(X) \arrow[d,"\mathcal{C}_{n+1}(\xi_m^X)"] 
\\
& 
X^{m+1} \arrow[r,"f_m"] 
& 
X^m.
\end{tikzcd}
\end{equation}
Here, $s: X^{m+1} \to Q_{m}$ exists due to the universal property of the pullback and we have $\mathbbm{1}_{X^{m+1}} = q_m s$. Therefore,
$$
\mathbbm{1}_{X^{m+1}} = q_m s \simeq \left(\mathcal{B}_{n+1}(\pi_{m+1}) \hspace{0.5mm} \Theta_m \right)s = \mathcal{B}_{n+1}(\pi_{m+1}) \left(\Theta_m s\right).
$$
Hence, $\Theta_m s$ is a homotopy section of $\mathcal{B}_{n+1}(\pi_{m+1})$. But since $\mathcal{B}_{n+1}(\pi_{m+1})$ is a fibration because of Proposition~\ref{fibration}, $\Theta_m s$ gives a section of $\mathcal{B}_{n+1}(\pi_{m+1})$. So by Proposition~\ref{chartcm} we get the inequality $\dTC_{m+1}(X) \leq n = \dcat(X^m)$. The reverse inequality $\dcat(X^m) \le \dTC_{m+1}(X)$ follows from Proposition~\ref{ineq2}.
\end{proof}

\begin{remark}
In view of Proposition~\ref{hspaces}, the lower bound of $\dTC_m$ in Proposition~\ref{ineq2} is sharp for CW $H$-spaces for each $m \ge 2$.
\end{remark}

\begin{cor}
If $G$ is a topological group, in particular a Lie group, then for all $m \ge 1$, we have $\dTC_{m+1}(G) = \dcat(G^m)$. 
\end{cor}

\begin{ex}\label{exmp}
For $k = 1,3,$ and $7$, we obtain $\dTC_m(S^k) = \dcat((S^k)^{m-1}) = m-1$ from Theorem~\ref{hspaces} and~\cite[Proposition 6.7]{DJ}. 
\end{ex}

\begin{ex}\label{tori}
For the $n$-torus $T^n$, due to Theorem~\ref{hspaces} and~\cite[Proposition 6.7]{DJ},
$$
\dTC_m(T^n) = \dcat((S^1)^{n(m-1)}) = n(m-1).
$$
\end{ex}

Like in the case of~\cite[Corollary 3.5]{LS} for classical invariants, it follows from Theorem~\ref{hspaces} that 
$$
\dTC_m(X^k) = \dcat(X^{k(m-1)}) = \dTC_{k+1}(X^{m-1})
$$
if $X$ is an $H$-space. Thus, $\dTC(T^2) = 2 = \dTC_3(S^1)$ and $\dTC_4(T^2) = 6 = \dTC_3(T^3)$.

\section{Estimates and Computations}\label{estcomp}

It is well-known in the literature that Alexander--Spanier cohomology coincides with singular cohomology on locally finite CW complexes. Hence, for all the spaces discussed in this section, Theorem~\ref{boundtcm} will hold in singular cohomology as well. So from now onwards, we will work only with singular cohomology groups.

\subsection{$\dTC_m$ of surfaces}

\begin{prop}\label{surfacebound}
Let $X$ be a finite CW complex and $m\ge 2$. If $H^d(X;\Q)\ne 0$ for some $d\ge 1$, then $\dTC_m(X)\ge m-1$.
\end{prop}

\begin{proof}
Let us choose $v \in H^d(X;\Q)$ such that $v \neq 0$. Due to Proposition~\ref{useful}, there exists an element $w \in H^d(SP^{(m-1)!}(X);\Q)$ such that $\pa_{m-1}^*(w) = v$, where $\pa_{m-1}:X \to SP^{(m-1)!}(X)$ is the diagonal embedding. For $1 \le i \le m$, let $\text{proj}_i : X^m \to X$ be the projection onto the $i$-th factor. By functoriality, we get $r_i = SP^{(m-1)!}(\text{proj}_i) : SP^{(m-1)!}(X^m) \to SP^{(m-1)!}(X)$. For brevity, let $Y = SP^{(m-1)!}(X)$ and $Z = SP^{(m-1)!}(X^m)$. Then for each $i\le m$, the following diagram commutes.
\begin{equation}\label{sixes}
\begin{tikzcd}
Y
&
Z \arrow{l}[swap]{r_i}
&
Y \arrow{l}[swap]{\Delta_{m-1}}
\\ 
X \arrow{u}{\pa_{m-1}}
& 
X^m \arrow{l}{\text{proj}_i} \arrow{u}{\pa_{m-1}^{m}}
&
X. \arrow{l}{\Delta} \arrow{u}[swap]{\pa_{m-1}}
\end{tikzcd}
\end{equation}
Here, $\pa_{m-1}^m: X^m \to Z$ is the diagonal embedding, $\Delta$ is the diagonal map, and $\Delta_{m-1}=SP^{(m-1)!}(\Delta)$. For each fixed $1 \le i \le m-1$, $\text{proj}_i$ and $\text{proj}_m$ induce the map $\phi_i: H^d(X;\Q) \oplus H^d(X;\Q) \to H^d(X^m;\Q)$, and similarly, $r_i$ and $r_m$ induce $\psi_i: H^d(Y;\Q) \oplus H^d(Y;\Q) \to H^d(Z;\Q)$ defined, respectively, as
$$
\phi_i(x \oplus y) = \text{proj}_i^*(x) - \text{proj}_m^*(y) \hspace{5mm} \text{and} \hspace{5mm} \psi_i(a \oplus b) = r_i^*(a) - r_m^*(b).
$$
Let us denote $\phi_i(v \oplus v)$ by $\alpha_i$ and $\psi_i(w \oplus w)$ by $\alpha_i^*$. Then due to Diagram~\ref{sixes}, we have the following commutative diagram for each $1 \le i \le m-1$.
\begin{equation}\label{sevens}
\begin{tikzcd}[contains/.style = {draw=none,"\in" description,sloped}]
w \oplus w \ar[r,mapsto] \ar[d,contains] & \alpha_i^* \ar[d,contains] \ar[r,mapsto] & w-w = 0 \ar[d,contains]
\\
H^{d}(Y;\Q)  \oplus H^{d}(Y;\Q)  \arrow{r}{\psi_i}  \arrow[swap]{d}{\pa_{m-1}^* \oplus \hspace{1mm} \pa_{m-1}^* \hspace{1mm}}
&
H^d(Z;\Q)  \arrow{r}{\Delta_{m-1}^*} \arrow[swap]{d}{(\pa_{m-1}^{m})^*}
&
H^d(Y;\Q) \arrow[swap]{d}{\pa_{m-1}^*}
\\
H^{d}(X;\Q)  \oplus H^{d}(X;\Q)  \arrow{r}[swap]{\phi_i} 
&
H^d(X^m;\Q)  \arrow{r}[swap]{\Delta^*} 
&
H^d(X;\Q).
\\
v \oplus v \ar[r,mapsto] \ar[u,contains] & \alpha_i \ar[u,contains] \ar[r,mapsto] & v-v = 0 \ar[u,contains]
\end{tikzcd}
\end{equation}
Diagram~\ref{sevens} can be explained as follows. Let $(\Delta^*\phi_i)_{|j}$ (resp. $(\Delta_{m-1}^*\psi_i)_{|j}$) be the restriction of $\Delta^*\phi_i$ (resp. $\Delta_{m-1}^*\psi_i$) to the $j$-th factor $H^d(X;\Q)$ (resp. $H^d(Y;\Q)$). Each $(\Delta^*\phi_i)_{|j}$ is an isomorphism because it contributes a copy of either $H^d(X;\Q)$ or $-H^d(X;\Q)$, where $-G=\{-g\mid g\in G\}$ for an abelian group $G$. In particular, $(\Delta^*\phi_i)_{|j}(v)=(-1)^{j+1}v$. Since the above diagram remains commutative when restricted on the top and bottom to $(\Delta_{m-1}^*\psi_i)_{|j}$ and $(\Delta^*\phi_i)_{|j}$, respectively, and since $\Delta^*\phi_i(v\oplus v)=v-v$~\cite{Ru}, it follows by the commutativity of Diagram~\ref{sevens} that $\Delta_{m-1}^*\psi_i(w\oplus w)=w-w=0$.
Note that due to~\cite[Proposition 3.5]{Ru}, we have 
$$
\alpha_1 \smile \cdots \smile \alpha_{m-1} \neq 0.
$$
Hence, taking $k_i = d$ for all $1 \le i \le m-1$, we have $\alpha_i^* \in H^{k_i}(SP^{(m-1)!}(X^m);\Q)$ as in the statement of Theorem~\ref{boundtcm}. Therefore, $\dTC_m(X) \ge m-1$.
\end{proof}

\begin{cor}\label{surfacebound2}
If $X$ is a closed orientable manifold, or if $X$ is a closed non-orientable surface of genus $> 1$, then $\dTC_m(X) \ge m-1$ for all $m \ge 2$.
\end{cor}
\begin{proof}
If $X$ is a closed orientable $n$-manifold, then $H^n(X;\Q)\ne 0$. If $X$ is a closed non-orientable surface of genus $> 1$, then $H^1(X;\Q)\ne 0$. Hence in both these situations, we can apply Proposition~\ref{surfacebound} to get $\dTC_m(X)\ge m-1$.
\end{proof}

The following simple consequence generalizes Example~\ref{exmp}.

\begin{cor}\label{oddsp}
For any $k \ge 1$ and $m \ge 2$, $\dTC_m(S^{2k-1}) = m-1$.
\end{cor}

\begin{proof}
From~\cite[Section 4]{Ru}, $\TC_m(S^{2k-1}) = m-1$. So, we obtain 
$$
m-1 \le \dTC_m(S^{2k-1}) \le \TC_m(S^{2k-1}) = m-1
$$
 due to Propositions~\ref{surfacebound} and~\ref{ineq1}.
\end{proof}

\begin{remark}
In view of Corollary~\ref{oddsp}, the lower bounds obtained in Theorem~\ref{boundtcm} and Proposition~\ref{surfacebound} are sharp for all $m \ge 2$. Also, the inequality in Proposition~\ref{ineq3} can be strict for each $m \ge 2$: this is due to Corollary~\ref{oddsp} and~\cite[Proposition 6.7]{DJ}. Furthermore, the inequality in Proposition~\ref{nondec} can be strict, and $\{\dTC_m(X)\}_{m\ge 2}$ can be a strictly increasing linear sequence.
\end{remark}

The following statement for closed orientable surfaces $\Sigma_g$ of genus $g \ge 2$ generalizes~\cite[Proposition 6.9]{DJ} for all $m,g \ge 2$.

\begin{prop}\label{orientsurface}
For any $m,g \ge 2$, $\dTC_m(\Sigma_g) = 2m$.
\end{prop}

\begin{proof}
Let us fix some $m,g \ge 2$. Let $a_i, b_i \in H^1(\Sigma_g;\Q)$ be the generators for $1 \le i \le g$ such that $a_i b_j = a_i a_j = b_i b_j = 0$ for $i \neq j$ and $a_ib_i\ne 0$ and $a_i^2 = b_i^2 = 0$ for each $i$. Due to Proposition~\ref{useful}, there exist $a_i^*,b_i^* \in H^1(SP^{(2m)!}(\Sigma_g);\Q)$ such that $\pa_{2m}^*(a_i^*) = a_i$ and $\pa_{2m}(b_i^*) = b_i$. We consider the projections $\text{proj}_i:\Sigma_g^m\to \Sigma_g$ onto the $i$-th factor and we let $r_i=SP^{(2m)!}(\text{proj}_i)$. For brevity, take $X=\Sigma_g$, $Y = SP^{(2m)!}(\Sigma_g)$ and $Z =SP^{(2m)!}(\Sigma_g^m)$. For each $2 \le i \le m$, let us define homomorphisms $\phi_i : H^1(X;\Q) \oplus H^1(X;\Q) \to H^1(X^m;\Q)$ and $\Phi_i : H^1(Y;\Q) \oplus H^1(Y;\Q) \to H^1(Z;\Q)$ as follows.
$$
\phi_i(c \oplus d) = \text{proj}_1^*(c) - \text{proj}_i^*(d) \hspace{5mm} \text{and} \hspace{5mm} \Phi_i(u \oplus v) = r_1^*(u) - r_i^*(v).
$$
For each $2 \le i \le m$, let $\alpha_i = \phi_i(a_1 \oplus a_1)$ and $\beta_i = \phi_i(b_1 \oplus b_1)$, and $\gamma_1 = \phi_2(a_2 \oplus a_2)$ and $\gamma_2 = \phi_2(b_2 \oplus b_2)$. Similarly, let $\alpha_i^* = \Phi_i(a_1^* \oplus a_1^*)$ and $\beta_i^* = \Phi_i(b_1^* \oplus b_1^*)$, and $\gamma_1^* = \Phi_2(a_2^* \oplus a_2^*)$ and $\gamma_2^* = \Phi_2(b_2^* \oplus b_2^*)$. For each $i$, we have the following commutative diagram.
\begin{equation}\label{eights} 
\begin{tikzcd}[contains/.style = {draw=none,"\in" description,sloped}]
H^{1}(Y;\Q)  \oplus H^{1}(Y;\Q)  \arrow{r}{\Phi_i}  \arrow[swap]{d}{\pa_{2m}^* \oplus \hspace{1mm} \pa_{2m}^* \hspace{1mm}}
&
H^1(Z;\Q)  \arrow{r}{\Delta_{2m}^*} \arrow[swap]{d}{(\pa_{2m}^{m})^*}
&
H^1(Y;\Q) \arrow[swap]{d}{\pa_{2m}^*}
\\
H^{1}(X;\Q)  \oplus H^{1}(X;\Q)  \arrow{r}{\phi_i} 
&
H^1(X^m;\Q)  \arrow{r}{\Delta^*} 
&
H^1(X;\Q).
\end{tikzcd}
\end{equation}
We note that $(\pa_{2m}^{m})^*(\alpha_i^*) = \alpha_i$ and $(\pa_{2m}^{m})^*(\beta_i^*) = \beta_i$ for all $2 \le i \le m$, and $(\pa_{2m}^{m})^*(\gamma_j^*) = \gamma_j$ for $j= 1,2$. Also, we have
$$
\Delta^*(\alpha_i) = a_1 - a_1 = 0 = b_1 - b_1 = \Delta^*(\beta_i), \hspace{2mm} \Delta^*(\gamma_1) = a_2-a_2=0=b_2-b_2=\Delta^*(\gamma_2),
$$
$$
\textup{and} \hspace{5mm} \alpha_2 \smile \cdots \smile \alpha_m \smile \beta_2 \smile \cdots \smile \beta_m \smile \gamma_1 \smile \gamma_2 \ne 0
$$
due to the proof of~\cite[Proposition 3.2]{GGGHMR}. Hence, by arguments similar to those used in Proposition~\ref{surfacebound}, we conclude that $\alpha_i^* ,\beta_i^*,\gamma_j^* \in \Ker(\Delta_{2m}^*)$ for $2 \le i \le m$ and $j = 1,2$. Thus, the $2m$ cohomology classes are as in Theorem~\ref{boundtcm}. Therefore, we have $\dTC_m(\Sigma_g) \ge 2m$. But $\TC_m(\Sigma_g) = 2m$ by~\cite[Proposition 3.2]{GGGHMR}. Hence, $\dTC_m(\Sigma_g) = 2m$ follows from Proposition~\ref{ineq1}. 
\end{proof}

\subsection{$\dTC_m$ of spheres and their products}
\begin{prop}\label{evensp}
For any $k \ge 1$ and $m \ge 2$, $\dTC_m(S^{2k}) = m$.
\end{prop}

\begin{proof}
Let $u \in H^{2k}(S^{2k};\Q)$ such that $u \neq 0$. By Proposition~\ref{useful}, there exists $w \in H^{2k}(SP^{m!}(S^{2k});\Q)$ such that $\pa_m^*(w) = u$. For each $1 \le i \le m$, we let $r_i=SP^{m!}(\text{proj}_i)$ for the projection $\text{proj}_i:(S^{2k})^m\to S^{2k}$ onto the $i$-th factor. For $X=S^{2k}$, define a map
$$
\phi: \bigoplus_{i=1}^m H^{2k}(X;\Q) \to H^{2k}(X^m;\Q) 
$$
using the homomorphisms $\text{proj}_i^*$ as 
$$
\phi\left(a_1 \oplus \cdots \oplus a_m\right) = \bigoplus_{i=1}^{m-1} \text{proj}_i^* \left(a_i \right) - \text{proj}_m^*\left(a_m\right).
$$
Again, we let $Y = SP^{m!}(X)$ and $Z = SP^{m!}(X^m)$. Similarly, define a map 
$$
\psi: \bigoplus_{i=1}^m H^{2k}(Y;\Q) \to H^{2k}(Z;\Q) 
$$
using the homomorphisms $r_i^*$ as 
$$
\psi\left(b_1 \oplus \cdots \oplus b_m\right) = \bigoplus_{i=1}^{m-1} r_i^* \left(b_i \right) - r_m^*\left(b_m\right).
$$
Let $\alpha = \phi(u \oplus \cdots \oplus u \oplus (m-1)u)$ and $\alpha^* = \psi(w \oplus \cdots \oplus w \oplus (m-1)w)$. Then we get the following commutative diagram.
\begin{equation}\label{nines}
\begin{tikzcd}[contains/.style = {draw=none,"\in" description,sloped}]
\bigoplus_{i=1}^{m-1} w \oplus (m-1)w \ar[r,mapsto] \ar[d,contains] & \alpha^* \ar[d,contains] \ar[r,mapsto] & (m-1)w-(m-1)w = 0 \ar[d,contains]
\\
\bigoplus_{i=1}^m H^{2k}(Y;\Q) \arrow{r}{\psi}  \arrow[swap]{d}{\bigoplus_{i=1}^m \pa_{m}^* \hspace{1mm}}
&
H^{2k}(Z;\Q)  \arrow{r}{\Delta_{m}^*} \arrow[swap]{d}{(\pa_{m}^{m})^*}
&
H^{2k}(Y;\Q) \arrow[swap]{d}{\pa_{m}^*}
\\
\bigoplus_{i=1}^m H^{2k}(X;\Q)   \arrow{r}{\phi} 
&
H^{2k}(X^m;\Q)  \arrow{r}{\Delta^*} 
&
H^{2k}(X;\Q).
\\
\bigoplus_{i=1}^{m-1} u \oplus (m-1)u \ar[r,mapsto] \ar[u,contains] & \alpha \ar[u,contains] \ar[r,mapsto] & (m-1)u-(m-1)u = 0 \ar[u,contains]
\end{tikzcd}
\end{equation}
Clearly, $(\pa_m^m)^*(\alpha^*) = \alpha$. Due to the computations done in~\cite[Section 4]{Ru}, $\alpha^m \neq 0$. Therefore, taking $k_i = 2k$ and $\alpha_i^* = \alpha^*$ for all $1 \le i \le m$ in the statement of Theorem~\ref{boundtcm}, we get $\dTC_m(S^{2k}) \ge m$. But $\TC_m(S^{2k}) = m$ due to~\cite[Section 4]{Ru}. Hence, $\dTC_m(S^{2k}) = m$ follows from Proposition~\ref{ineq1}.  
\end{proof}

For cases $g = 0,1$, since $\Sigma_0=S^2$ and $\Sigma_1=T^2$, we have from Proposition~\ref{evensp} and Example~\ref{tori}, respectively, that $\dTC_m(\Sigma_0) = m$ and $\dTC_m(\Sigma_1) = 2(m-1)$.

\begin{remark}
In view of Proposition~\ref{evensp} and~\cite[Proposition 6.7]{DJ}, the upper bound of $\dTC_m$ in Proposition~\ref{ineq3} is sharp for all $m \ge 2$. Also, the inequalities in Propositions~\ref{ineq2} and~\ref{surfacebound} and Corollary~\ref{surfacebound2} can be strict for all $m \ge 2$. 
\end{remark}

Example~\ref{tori} is recovered from the following more general result which, in particular, extends~\cite[Proposition 6.8]{DJ} to finite products of spheres.

\begin{prop}\label{spprod}
For any $m \ge 2$ and $n \ge 1$, 
$$
\dTC_m\left(\prod_{j=1}^{n} S^{k_j}\right) = n(m-1) + l_n,
$$
where $l_n$ is the number of $k_j$ that are even.
\end{prop}

\begin{proof}
For $n = 1$, this holds due to Corollary~\ref{oddsp} and Proposition~\ref{evensp}. First, we prove this result for $n = 2$. Let us fix some $m \ge 2$ and take $X = S^{k_1} \times S^{k_2}$. For $j = 1,2$, define a map $\pi_j : X \to S^{k_j}$ as the projection onto the $j$-th coordinate, $\Delta_j : S^{k_j} \hookrightarrow (S^{k_j})^m$ as the diagonal embedding, and for each $1 \le t \le m$, define a map $p_t^j:(S^{k_j})^m \to S^{k_j}$ as the projection onto the $t$-th copy of $S^{k_j}$. For each $t$, define $r_t = p_t^1\times p_t^2: (S^{k_1})^m\times (S^{k_2})^m\to S^{k_1}\times S^{k_2}= X$. Here, $(S^{k_1})^m\times (S^{k_2})^m$ is homeomorphic to $X^m$ due to some reordering of the factors of the cartesian product, so we shall write $r_t:X^m\to X$ for simplicity. Let 
$$
d_j = \begin{cases}
m & : k_j \text{ is even} \\
m-1 & : k_j \text{ is odd}
\end{cases} \hspace{5mm} \text{ and } \hspace{5mm} e_j = \begin{cases}
m & : k_j \text{ is even} \\
2 & : k_j \text{ is odd}
\end{cases}
$$
Let $d = d_1 + d_2$. Define $\pa: X \to Y$ and $\pa_m: X^m \to Z$ as the diagonal embeddings, where $Y = SP^{d!}(X)$ and $Z = SP^{d!}(X^m)$. Define $R_t^j = SP^{d!}(r_t^j) : Z \to Y$. For each $j =1,2$ and all $1 \le t \le m$, let $(\text{proj}_t ^j)^*:H^{k_j}(S^{k_j};\Q)\to H^{k_j}((S^{k_j})^m;\Q)$ be the maps induced by the projections $\text{proj}_t^j:(S^{k_j})^m\to S^{k_j}$. For each fixed $j=1,2$, we define homomorphisms 
$$
\phi_i^j : \bigoplus_{t=1}^{e_j} H^{k_j}(S^{k_j};\Q) \to H^{k_j}((S^{k_j})^m;\Q)
$$ 
for all $1 \le i \le d_j$ as
$$
\phi_i^j \left(\bigoplus_{t =1}^{e_j} a_t \right) = \begin{cases}
\bigoplus_{t=1}^{m-1} (\text{proj}_t^j)^* \left(a_t \right) - (\text{proj}_m^j)^*\left(a_m\right) & : k_j \text{ is even}  \\
(\text{proj}_i^j)^*\left(a_1\right) - (\text{proj}_m^j)^* \left(a_2 \right) & : k_j \text{ is odd} 
\end{cases}
$$
So, if $k_j$ is even, then the $m$ homomorphisms $\phi^j_i$ are all the same as the homomorphism $\phi$ from Proposition~\ref{evensp}. Choose $u_j \in H^{k_j}(S^{k_j};\Q)$ such that $u_j \neq 0$. Then, it follows from Diagrams~\ref{sevens} and~\ref{nines} (depending on the parity of $k_j$) that 
\begin{equation}\label{tens}
\left(\Delta_j^* \phi_i^j \right) \left(\bigoplus_{t =1}^{e_j-1} u_j \oplus (e_j-1)u_j\right)= \left(e_j -1\right)u_j - \left(e_j -1\right)u_j = 0
\end{equation}
for each $i$ and $j$. In the same spirit, for each $j=1,2$, we define homomorphisms
$$
\psi_i^j : \bigoplus_{t=1}^{e_j} H^{k_j}(X;\Q) \to H^{k_j}(X^m;\Q)
$$ 
for all $1 \le i \le d_j$ as
$$
\psi_i^j \left(\bigoplus_{t =1}^{e_j} b_t \right) = \begin{cases}
\bigoplus_{t=1}^{m-1} r_t^* \left(b_t \right) - r_m^*\left(b_m\right) & : k_j \text{ is even}  \\
r_i^*\left(b_1\right) - r_m^* \left(b_2 \right) & : k_j \text{ is odd} 
\end{cases}
$$
Finally, we for each fixed $j=1,2$, we define homomorphisms
$$
\Psi_i^j : \bigoplus_{t=1}^{e_j} H^{k_j}(Y;\Q) \to H^{k_j}(Z;\Q)
$$ 
for all $1 \le i \le d_j$ as
$$
\Psi_i^j \left(\bigoplus_{t =1}^{e_j} z_t \right) = \begin{cases}
\bigoplus_{t=1}^{m-1} R_t^* \left(z_t \right) - R_m^*\left(z_m\right) & : k_j \text{ is even}  \\
R_i^*\left(z_1\right) - R_m^* \left(z_2 \right) & : k_j \text{ is odd} 
\end{cases}
$$
Let $\eta = SP^{d!}(\Delta_1 \times \Delta_2): Y \to Z$. Then for any fixed $j\in\{1,2\}$, the following diagram commutes for each $1 \le i \le d_j$.
\begin{equation}\label{elevens}
\begin{tikzcd}[contains/.style = {draw=none,"\in" description,sloped}]
\bigoplus_{t=1}^{e_j} H^{k_j}(Y;\Q) \arrow{r}{\Psi_i^j}  \arrow{d}{\bigoplus_{t=1}^{e_j} \pa^* \hspace{1mm}}
&
H^{k_j}(Z;\Q)  \arrow{r}{\eta^*} \arrow[swap]{d}{\pa_{m}^*}
&
H^{k_j}(Y;\Q) \arrow[swap]{d}{\pa^*}
\\
\bigoplus_{t=1}^{e_j} H^{k_j}(X;\Q) \arrow{r}{\psi_i^j} 
&
H^{k_j}(X^m;\Q)  \arrow{r}{(\Delta_1 \times \Delta_2)^*}  
&
H^{k_j}(X;\Q)
\\
\bigoplus_{t=1}^{e_j} H^{k_j}(S^{k_j};\Q)  \arrow{r}{\phi_i^j} \arrow[swap]{u}{\bigoplus_{t=1}^{e_j} \pi_j^* \hspace{1mm}}
&
H^{k_j}((S^{k_j})^m;\Q)  \arrow{r}{\Delta_j^*}  \arrow{u}{(\pi_j^{m})^*}
&
H^{k_j}(S^{k_j};\Q). \arrow{u}{\pi_j^*}
\end{tikzcd}
\end{equation}
Let $y_j = \pi_j^*(u_j) \in H^{k_j}(X;\Q)$. For each $1 \le p \le d_1$ and $1 \le q \le d_2$, let us write
$$
\alpha_p = \psi_p^1\left(\bigoplus_{t=1}^{e_1-1} y_1 \oplus (e_1-1)y_1\right) \hspace{5mm} \text{and} \hspace{5mm} \beta_q = \psi_q^2\left(\bigoplus_{t=1}^{e_2-1} y_2 \oplus (e_2-1)y_2\right).
$$
We note that if $k_1$ is even, then $\alpha_1=\alpha_2=\cdots=\alpha_m$, and if $k_2$ is even, then $\beta_1=\beta_2=\cdots=\beta_m$. It follows from the proof of~\cite[Theorem 3.10]{BGRT} that 
$$
\alpha_1 \smile \cdots \smile \alpha_{d_1} \smile \beta_1 \smile \cdots \smile \beta_{d_2} \neq 0.
$$
Also, in view of Equation~\ref{tens} and Diagram~\ref{elevens}, we have for each $p$ and $q$ that
\begin{equation}\label{lasts}
(\Delta_1 \times \Delta_2)^*(\alpha_p) = 0 = (\Delta_1 \times \Delta_2)^*(\beta_q).
\end{equation}
Due to Proposition~\ref{useful}, there exist some $y_j^* \in H^{k_j}(Y;\Q)$ such that $\pa^*(y_j^*) = y_j$. For each $1 \le p \le d_1$ and $1 \le q \le d_2$, let
$$
\alpha_p^* = \Psi_p^1\left(\bigoplus_{t=1}^{e_1-1} y_1^* \oplus (e_1-1)y_1^*\right) \hspace{5mm} \text{and} \hspace{5mm} \beta_q^* = \Psi_q^2\left(\bigoplus_{t=1}^{e_2-1} y_2^* \oplus (e_2-1)y_2^*\right).
$$
Then, $\pa_m^*(\alpha_p^*) = \alpha_p$ and $\pa_m^*(\beta_q^*) = \beta_q$ by Diagram~\ref{elevens}. As explained after Diagram~\ref{sevens}, the top (resp. middle and bottom) row of Diagram~\ref{elevens} when restricted on the left to either of the components $H^{k_j}(Y;\Q)$ (resp. $H^{k_j}(X;\Q)$ and $H^{k_j}(S^{k_j};\Q)$) is an isomorphism. Thus, using Equation~\ref{lasts}, we get
$$
\eta^*(\alpha_p^*) = (e_1-1)y_1^*-(e_1-1)y_1^* = 0= (e_2-1)y_2^*-(e_2-1)y_2^*= \eta^*(\beta_q^*)
$$
for all $1 \le p \le d_1$ and $1 \le q \le d_2$. We note that $\Delta_1 \times \Delta_2: X \to X^m$ is the diagonal map (more precisely, it is the composition of the diagonal map $X\to X^m$ with a factor-reordering homeomorphism of the cartesian product $X^m$). Hence, $\dTC_m(S^{k_1} \times S^{k_2}) \ge d = 2(m-1) + l_2$ follows from Theorem~\ref{boundtcm}, where $l_2 \in \{0,1,2\}$ depending on the parity of $k_1$ and $k_2$. 

In general, for any $n \ge 3$, the same technique can be used to get the inequality 
$$
\dTC_m\left(\prod_{j=1}^{n} S^{k_j}\right) \ge n(m-1) + l_n.
$$
To see this, let $A= S^{k_1}\times S^{k_2}$ and consider $X=A\times S^{k_3}$. Following the above proof for $A$, we get $2(m-1) + l_2$ many $m$-th zero-divisors of $X$ that come from $A$. We also get either $m-1$ or $m$ many $m$-th zero-divisors of $X$ that come from $S^{k_3}$ depending on the parity of $k_3$. Then proceeding as in the above case for $n=2$, we obtain $\dTC_m(A\times S^{k_3}) \ge 3(m-1) + l_3.$ Using this proof for the product of three spheres, we can do this for the product of four spheres and so on for the product of $n$ spheres for each $n\ge 4$ by taking $A= \prod_{j=1}^{n-1} S^{k_j}$ and $X=A\times S^{k_n}$.

In any case, due to~\cite[Corollary 3.12]{BGRT}, $\TC_m(\prod_{j=1}^{n} S^{k_j}) = n(m-1) + l_n$. Therefore, $\dTC_m(\prod_{j=1}^{n} S^{k_j}) = n(m-1) + l_n$ follows from Proposition~\ref{ineq1}.
\end{proof}

For cases $g =0,1$, this gives $\dTC_m(\Sigma_0 \times \prod_{j=1}^{n} S^{k_j}) = m(n+1) - n + l_n$ and $\dTC_m(\Sigma_1 \times \prod_{j=1}^{n} S^{k_j})= (m-1)(n+2) + l_n$.

\begin{remark}
Due to Proposition~\ref{spprod}, the inequality in Proposition~\ref{ineq4} can be strict for all $m \ge 2$ for non-contractible spaces $X$ and $Y$. 
\end{remark}

The techniques used in the proofs of Propositions~\ref{spprod} and~\ref{orientsurface} can be used more generally to obtain the following.

\begin{prop}\label{proddiff}
For any $m,g \ge 2$ and $n \ge 1$, 
$$
\dTC_m\left(\Sigma_g \times \prod_{j=1}^{n} S^{k_j}\right) = m(n+2) - n + l_n,
$$
where $l_n$ is the number of $k_j$ that are even.
\end{prop}

\begin{proof}
Take $A = \Sigma_g$ and $B = \prod_{j=1}^{n} S^{k_j}$. By Proposition~\ref{orientsurface}, there are $2m$ cohomology classes, say $\alpha_i \in H^1(A^m;\Q)$, that vanish under $\Delta_1^*: H^1(A^m;\Q) \to H^1(A;\Q)$. By Proposition~\ref{spprod}, there are $d = n(m-1)+l_n$ cohomology classes, say $\beta_k \in H^*(B^m;\Q)$, that vanish under the homomorphism $\Delta_2^*: H^*(B^m;\Q) \to H^*(B;\Q)$ induced in the $k_j$-th cohomology by the diagonal map $\Delta_2:B\to B^m$ if $\beta_k$ is a class coming from the cohomology of $(S^{k_j})^m$.
In the proof of Proposition~\ref{spprod} for the product of two spheres, we replace the roles of the spheres $S^{k_1}$ and $S^{k_2}$ with $A$ and $B$, respectively, and define $X = A \times B$. Then from a direct analog of Diagram~\ref{elevens} for $A$ and $B$, we get images $\alpha_i^* \in H^1(X^m;\Q)$ and $\beta_k^* \in H^*(X^m;\Q)$ such that 
$$
(\Delta_1 \times \Delta_2)^*(\alpha_i^*) = 0 = (\Delta_1 \times \Delta_2)^*(\beta_k^*).
$$
Then using Proposition~\ref{useful}, we get $\alpha_i^{**} \in H^1(Z;\Q)$ and $\beta_k^{**} \in H^*(Z;\Q)$ satisfying 
$$
\eta^*(\alpha_i^{**}) = 0 = \eta^*(\beta_k^{**})
$$
such that $\pa_m^*(\alpha_i^{**}) = \alpha_i^*$ and $\pa_m^*(\beta_k^{**}) = \beta_k^*$ for all $1 \le i \le 2m$ and $1 \le k \le d$. We recall that here, $Z=SP^{(2m+d)!}(X^m)$, $\pa_m:X^m\to Z$ is the diagonal embedding, $\Delta_1:A\to A^m$ and $\Delta_2:B\to B^m$ are diagonal maps, and $\eta=SP^{(2m+d)!}(\Delta_1\times\Delta_2)$. From the proof of~\cite[Theorem 3.10]{BGRT}, we deduce that
$$
\alpha_1^* \smile \cdots \smile \alpha_{2m}^* \smile \beta_1^* \smile \cdots \smile \beta_d^*  \neq 0.
$$
Hence, $\alpha_i^{**},\beta_i^{**} \in H^*(Z;\Q)$ are $2m + d$ number of cohomology classes satisfying the conditions of Theorem~\ref{boundtcm}. Therefore, $\dTC_m(X) \ge 2m + d = m(n+2) - n + l_n$. Note that $X^m$ is a normal space, $\TC_m(A) = 2m$ by~\cite[Proposition 3.2]{GGGHMR}, and $\TC_m(B) = n(m-1) + l_n$ by~\cite[Corollary 3.12]{BGRT}. So, 
$$
m(n+2) - n + l_n \le \dTC_m(X) \le \TC_m(X) \le \TC_m(A) + \TC_m(B) = m(n+2) - n + l_n,
$$
where the third inequality is due to~\cite[Proposition 3.11]{BGRT}.
\end{proof}

For classical $\TC_m$, the above set of inequalities immediately implies that
$$
\TC_m\left(\Sigma_g \times \prod_{j=1}^{n} S^{k_j}\right) = m(n+2) - n + l_n.
$$

\begin{remark}\label{coincide}
Summarizing the results of this paper so far, we see that the sequences $\{\dTC_m(X)\}_{m\ge 2}$ and $\{\TC_m(X)\}_{m\ge 2}$ coincide for each of the following classes of closed manifolds $X$: spheres, finite products of spheres, closed orientable surfaces $\Sigma_g$, and products of $\Sigma_g$ with finite products of spheres. In particular, the upper bound of $\dTC_m$ in Proposition~\ref{ineq1} is sharp for all these spaces.
\end{remark}

\section{The Analog Invariants}\label{kwnew}

For a metric space $(X,d)$ and $n \ge 1$, let us consider the following quotient map from~\cite[Definition 2.1]{KW}:
$$
q : \bigsqcup_{i=1}^{n} \left(X^i \times \Delta^{i-1} \right) \to  \bigsqcup_{i=1}^{n} \left(X^i \times \Delta^{i-1} \right) / \sim  \hspace{2mm}= :\mathscr{P}_n(X).
$$
Here, $(x_1,\ldots,x_n,t_1,\ldots,t_n)\sim (x_{\sigma(1)},\ldots,x_{\sigma(n)},t_{\sigma(1)},\ldots,t_{\sigma(n)})$ for each $\sigma\in S_n$, $(x_1,\ldots,x_{n-1},x_n,t_1,\ldots,t_{n-1},t_n)\sim (x_1,\ldots,x_{n-1},t_1,\ldots,t_{n-1}+t_n)$ if $x_{n-1}=x_n$, and $(x_1,\ldots,x_{n},t_1,\ldots,t_{n})\sim (x_1,\ldots,x_{n-1},t_1,\ldots,t_{n-1})$ if $t_n=0$.

Let $\mathscr{T}_1$ be the quotient topology on $\mathscr{P}_n(X)$. For any fibration $p:X \to B$, define $\mathscr{P}_n(p): = \{\mu \in \mathscr{P}_n(X) \mid |\hspace{0.3mm}p(\supp(\mu))| = 1\}$. See that as sets, $\mathscr{P}_n(X) = \mathcal{B}_n(X)$ and $\mathscr{P}_n(p) = X_n(p)$. Let us choose some $\mu = \lambda_x \hspace{0.5mm} x \in \mathcal{B}_n(X)$ and fix it. Let $\delta = \text{min}\{d(x,y) \mid x,y \in \text{supp}(\mu), x \ne y\}$. Then for $\epsilon < \delta/2$, the collection of the sets
$$
\mathscr{U}(\mu,\epsilon) = \left\{\sum s_y \hspace{0.5mm} y \hspace{1mm} \middle| \hspace{1mm} -\frac{\epsilon}{n} < \sum_{y \in B(x,\epsilon)} s_y -\lambda_x < \frac{\epsilon}{n} \right\}
$$
defines a local basis at $\mu$ in the topology induced by the Lévy--Prokhorov metric on $\mathcal{B}_n(X)$,~\cite{DJ} (see also the description of the Lévy--Prokhorov distance in Section~\ref{contmot}). These sets are also considered in~\cite[Construction 8.3]{KW}, and it follows from the proof of the first part of Lemma 8.4 in~\cite{KW} that the sets $\mathscr{U}(\mu,\epsilon)$ are open in the quotient topology on $\mathscr{P}_n(X)$. Hence, $\mathscr{T}_2 \subset \mathscr{T}_1$, where $\mathscr{T}_2$ denotes the Lévy--Prokhorov topology.
Therefore, the identity map $\mathcal{I}' : (\mathscr{P}_n(X),\mathscr{T}_1) \to (\mathcal{B}_n(X),\mathscr{T}_2)$ is continuous and thus, so is its restriction $\mathcal{I} : (\mathscr{P}_n(p),\mathscr{T}_1) \to (X_n(p),\mathscr{T}_2)$.

\subsection{Comparison between the invariants}\label{seclast}

For a fibration $p : E \to B$, its \textit{analog sectional category}, denoted $\text{asecat}(p)$, is the minimal number $n$ such that the map $\varphi_{n+1}(p): (\mathscr{P}_{n+1}(p), \mathscr{T}_1) \to B$, defined as $\varphi_{n+1}(p) : \mu \mapsto p(\supp(\mu))$, admits a section, see~\cite[Definition 5.1]{KW}.  

\begin{lemma}\label{asecat}
For any fibration $p: E \to B$, $\dsecat(p) \le \textup{asecat}(p)$. 
\end{lemma}

\begin{proof}
    Let $\text{asecat}(p) = n-1$. So, there exists a section $\kappa : B \to (\mathscr{P}_n(p), \mathscr{T}_1)$ of $\varphi_n(p)$. Consider $\kappa' = \mathcal{I} \hspace{0.5mm} \kappa: B \to (E_n(p),\mathscr{T}_2)$. Since $\mathcal{I}$ is continuous, so is $\kappa'$. It is clear that $\varphi_n(p) = \mathcal{B}_n(p) \hspace{0.5mm} \mathcal{I}$. Next, see that
    $$
    \mathbbm{1}_{B} = \varphi_n(p) \hspace{0.5mm} \kappa = \mathcal{B}_n(p) \hspace{0.5mm} \mathcal{I} \hspace{0.5mm} \kappa = \mathcal{B}_n(p) \hspace{0.5mm} \kappa'.
    $$
    So, $\kappa'$ is a section of $\mathcal{B}_n(p): (E_n(p),\mathscr{T}_1) \to B$. Thus, $\dsecat(p) \le n-1$.
\end{proof}

Since $\dsecat$,  $\text{asecat}$, and $\secat$ are all homotopy invariants, we have for any map $f:X \to Y$ the inequalities 
\begin{equation}\label{nnnew}
    \dsecat(f) \le \textup{asecat}(f) \le \secat(f),
\end{equation}
where the last inequality is because of~\cite[Corollary 5.5]{KW}.

\begin{cor}\label{acat}
    $\dcat(X) \le \acat(X)$ and $\dTC_m(X) \le \ATC_m(X)$ for all $m \ge 2$.
\end{cor}

\begin{proof}
    For the evaluation fibrations $p_0^X : P_0(X) \to X$ and $\pi_m^X:P(X) \to X^m$, we have $\dcat(X) = \dsecat(p_{0}^X)$ and $\dTC_m(X) = \dsecat(\pi_{m}^X)$ from Section~\ref{charrr}, and $\acat(X) = \text{asecat}(p_{0}^X)$ and $\ATC_m(X) = \text{asecat}(\pi_{m}^X)$ from~\cite[Definition 6.1]{KW}. Thus, the required inequalities follow directly from Lemma~\ref{asecat}.
\end{proof}

\begin{ex}\label{lateref}
    Corollary~\ref{acat},~\cite[Corollaries 6.6 and 7.3]{KW} imply that
\begin{enumerate}
\item $\dTC_m(\R P^n) \le \ATC_m(\R P^n) \le 2m+1$ for all $n \ge 1$, and
\item $\dTC_m(\Gamma) = \dTC_m(B\Gamma)  \le \ATC_m(B\Gamma) = |\Gamma| - 1$ for a finite group $\Gamma$.
\end{enumerate}
Item (1) significantly improves our bound from Corollary~\ref{tcmrpn}, while Item (2) gives $\dTC_m(\Z_2) = \dTC_m(\R P^{\infty}) = 1$ for each $m\ge 2$. 
\end{ex} 

We note that for the cases $m=3,4$, the upper bounds $\dTC_3(\R P^n)\le 3$ and $\dTC_4(\R P^n)\le 7$ obtained in Corollary~\ref{tcmrpn} are still better than the respective upper bounds $\dTC_3(\R P^n)\le 7$ and $\dTC_4(\R P^n)\le 9$ obtained in Example~\ref{lateref}.


\subsection{Completing some computations for $\acat$ and $\ATC_m$}

With the help of our computations from Sections~\ref{forhspaces} and~\ref{estcomp}, and~\cite[Section 6]{DJ}, we complete some computations of~\cite{KW} using Corollary~\ref{acat}.

\begin{remark}
    For a torsion-free discrete group $\Gamma$, $\ATC_m(B\Gamma) \le \cd(\Gamma^m)$ holds for all $m \ge 2$,~\cite[Corollary 7.7] {KW}. For closed orientable surfaces $\Sigma_g$ with $g \ge 1$, this gives $\ATC_m(\Sigma_g) \le 2m$. However, this is obvious in light of Equation~\ref{nnnew} and doesn't help much in finding the exact value $\ATC_m$ unless one gets into the technicalities of~\cite[Theorem 7.9]{KW} for cases $g \ge 2$. Also, for the $n$-torus $T^n$, the above inequality yields $\ATC_m(T^n) \le nm$. This upper bound is weaker than the bound $n(m-1)$ implied by Equation~\ref{nnnew}. In particular, it does not determine the value $\ATC_m(T^n)$.
\end{remark}

For non-aspherical spaces $X$, such as spheres $S^n$ for $n \ge 2$ and their products, there is no general recipe as such for computing $\acat(X)$ or $\ATC_m(X)$. 

We now show in the following two examples that the exact values of sequential analog topological complexities can be determined easily for various closed manifolds, most of which are not aspherical. 

\begin{ex}\label{exa}
    If $X$ is a sphere, a finite product of spheres, a closed orientable surface $\Sigma_g$, or a product of $\Sigma_g$ with a finite product of spheres, then in view of Remark~\ref{coincide}, Equation~\ref{nnnew}, and Corollary~\ref{acat}, we have
    $$
    \TC_m(X) = \dTC_m(X) \le \ATC_m(X) \le \TC_m(X).
    $$\end{ex}

\begin{ex}
    Similarly, if $X$ is a closed surface, a sphere, a finite product of spheres, or $\C P^n$, then due to~\cite[Propositions 6.1, 6.4, 6.6, and 6.7]{DJ}, we have
    $$
    \cat(X) = \dcat(X) \le \acat(X) \le \cat(X).
    $$\end{ex}

\begin{remark}
    We note that if $\Gamma$ is a torsion-free  discrete group, then due to \cite[Theorem 7.4]{KW}, $\acat(\Gamma)=\cd(\Gamma)$ holds. Hence, $\acat(T^n)=\cat(T^n)=n$ and $\acat(\Sigma_g)=\cat(\Sigma_g)=2$ for $g\ge 1$ are obtained directly from that.
\end{remark}

In light of the above examples, all the lower bounds in Corollary~\ref{acat} are sharp.

\section*{Acknowledgement}
The author would like to thank Alexander Dranishnikov for his kind guidance and various insightful and helpful discussions. The author is also greatful to the anonymous referee for several comments and suggestions that helped improve the exposition of this paper.


\end{document}